\documentclass[twoside,%pdftex%,11pt%
]{amsart}

  \usepackage[utf8]{inputenc}
  \usepackage{amsmath}
  \usepackage{amssymb}
  \usepackage{mathrsfs}
  \usepackage{bm}
  \usepackage{bbm}
  \usepackage{pifont}
  \usepackage{hyperref}
  
  \usepackage[svgnames,x11names]{xcolor} 

  \hypersetup{
    colorlinks, linkcolor={Chocolate4},
    citecolor={Chocolate4}, urlcolor={DarkOliveGreen4}
  }

  \usepackage{color}
  \usepackage{todonotes}

	\newtheoremstyle{slanted}% name
	{}%      Space above, empty = `usual value'
	{}%      Space below
	{\slshape}% Body font
	{}%         Indent amount (empty = no indent, \parindent = para indent)
	{\bfseries}% theo head font
	{.}%        Punctuation after theo head
	{ }% Space after theo head: \newline = linebreak
	{}%         theo head spec
	
	\theoremstyle{slanted}
	\newtheorem{theo}{Theorem}[section]
	\newtheorem{prop}[theo]{Proposition}

	\newtheorem{lemma}[theo]{Lemma}
	\newtheorem{definition}[theo]{Definition}
	\newtheorem{corollary}[theo]{Corollary}
	\newtheorem{remark}[theo]{Remark}
	\newtheorem{example}[theo]{Example}

	\newcommand{\egdef}{:=}
	\DeclareMathOperator{\Id}{Id}	
	\newcommand{\tend}[3][]{\xrightarrow[#2\to#3]{#1}}
	\newcommand{\EE}{\mathbb{E}}

	\newcommand{\ind}[1]{\mathbbmss{1}_{#1}} %needs bbm
	\newcommand{\ZZ}{\mathbb{Z}}

	\newcommand{\RR}{\mathbb{R}}
	\newcommand{\TT}{\mathbb{T}}
	\newcommand{\PP}{\mathbb{P}}
	\newcommand{\NN}{\mathbb{N}}
	\newcommand{\A}{\mathcal{A}}
	\newcommand{\B}{\mathcal{B}}
	\newcommand{\F}{\mathscr{F}}
	
	\renewcommand{\P}{\mathscr{P}}
	\newcommand{\Z}{\mathcal{Z}}
	\newcommand{\N}{\mathcal{N}}

	\newcommand{\PaP}{$\mathcal{P}a\mathcal{P}$}
	
\title{Poisson suspensions and SuShis}
\author{\'{E}lise Janvresse, Emmanuel Roy and Thierry de la Rue}
\address{\'Elise Janvresse: 
Laboratoire Amiénois de Mathématique Fondamentale et Appliquée, CNRS-UMR 7352, Université de Picardie Jules Verne, 33 rue Saint Leu, F80039 Amiens cedex 1,
France.}
\email{Elise.Janvresse@u-picardie.fr}
\address{Emmanuel Roy: Laboratoire Analyse, Géométrie et Applications, Université Paris 13 Institut Galilée,
99 avenue Jean-Baptiste Clément
F93430 Villetaneuse, France.}
\email{roy@math.univ-paris13.fr}
\address{Thierry de la Rue:
Laboratoire de Mathématiques Rapha\"el Salem,
Université de Rouen, CNRS,
Avenue de l'Université,
F76801 Saint \'Etienne du Rouvray, France.}
\email{Thierry.de-la-Rue@univ-rouen.fr}
\thanks{Research partially supported by French research group GeoSto
(CNRS-GDR3477)}

\begin{document}
\bibliographystyle{amsplain}

\maketitle

\begin{abstract}
In this paper, we prove that ergodic point processes with moments of all orders, driven by particular infinite measure preserving transformations, 
have to be a superposition of shifted Poisson processes. 
This rigidity result has a lot of implications in terms of joining and disjointness for the corresponding Poisson suspension. 
In particular, we prove that its ergodic self-joinings are Poisson joinings, which provides an analog, in the Poissonian context, of the GAG property for Gaussian dynamical systems.
\end{abstract}
\bigskip

\renewcommand{\abstractname}{Résumé}
\begin{center}
  \textbf{Suspensions de Poisson et SuShis}
\end{center}
\begin{abstract}
Dans cet article, nous démontrons qu'un processus ponctuel ergodique avec des moments de tous ordres, dirigé par une transformation pré\-ser\-vant une mesure infinie qui vérifie certaines propriétés, est nécessairement une superposition de processus de Poisson décalés. Ce résultat de rigidité a de nombreuses implications en termes de couplages et de disjonction pour la suspension de Poisson associée. En particulier, nous démontrons que ses autocouplages ergodiques sont des couplages poissoniens, obtenant ainsi un analogue, dans le contexte poissonien, de la propriété GAG des syst\`emes dynamiques gaussiens. 
\end{abstract}

\section{Introduction}

Central to probability theory are Gaussian and Poisson distributions.
In ergodic theory, they both play a particular role through canonical
constructions we briefly recall:
\begin{itemize}
\item Starting from a positive and finite symmetric Borel measure $\sigma$ on $\mathbb{T}$, there exists
a unique centered stationary real-valued Gaussian process $\left\{ X_{n}\right\} _{n\in\mathbb{Z}}$
whose coordinates admit $\sigma$ as spectral measure, that is
\[
\mathbb{E}\left[X_{0}X_{n}\right]=\widehat{\sigma}\left(n\right).
\]

\item Starting from a $\sigma$-finite dynamical system $\left(X,\mathcal{A},\mu,T\right)$,
we can build the Poisson suspension $\left(X^{*},\mathcal{A}^{*},\mu^{*},T_{*}\right)$,
which is the canonical space $\left(X^{*},\mathcal{A}^{*},\mu^{*}\right)$ of the Poisson point process of intensity $\mu$, 
enriched by the transformation
\[
T_{*}\left(\xi\right):=\xi\circ T^{-1}
\]
 (see below for a precise definition).
\end{itemize}
A striking theorem due to Foia\c{s} and Str$\breve{\mathrm{a}}$til$\breve{\mathrm{a}}$ (see \cite{FoiasStratila}) states that some measures $\sigma$
on $\mathbb{T}$, if appearing as spectral measure of some ergodic
stationary process, force the process to be Gaussian. This was considerably developed later (see \cite{LemParThou00Gausselfjoin} in particular) and
lead to some remarkable results.

In this paper, we obtain a Poisson counterpart of Foia\c{s}-Str$\breve{\mathrm{a}}$til$\breve{\mathrm{a}}$ theorem. 
We prove that some ergodic infinite measure preserving transformation,
taken as base system of an ergodic invariant point process with moments
of all orders, force the latter to be a superposition
of shifted Poisson point processes. 
% In case of non-Abelian actions, 
% we can even get rid of the superposition and obtain the Poisson process as only invariant point process.

\subsubsection*{Notations}
For any set $J$, we denote by $\#J$ the cardinality of $J$. If $\varphi$ is any measurable map from $(X,\A)$ to $(Y,\B)$, and if $m$ is a measure on $(X,\A)$, we denote by $\varphi_*(m)$ the pushforward measure of $m$ by $\varphi$.

\subsection{Random measures and point processes}

Let $X$ be a complete separable metric space and $\A$ be its Borel $\sigma$-algebra.  
Define $\widetilde{X}$ to be the space of \emph{boundedly finite measures} on $\left(X,{\A}\right)$, that is to say measures giving finite mass to any bounded Borel subset of $X$.
We refer to~\cite{DaleyVereJonesI} for the topological properties of $\widetilde{X}$. In particular, $\widetilde{X}$ can be turned into a complete separable metric space, and its Borel $\sigma$-algebra $\widetilde{\A}$ is generated by the maps $\widetilde X\ni \xi \mapsto \xi(A)\in\mathbb{R}_+\cup\left\{ +\infty\right\}$ for bounded $A\in\A$.

Let $X^{*}\subset \widetilde{X}$ be the subspace of simple counting measures, \textit{i.e.} whose elements are of the form 
\[
\xi=\sum_{i\in I}\delta_{x_{i}},
\]
where $I$ is at most countable, and $x_i\neq x_j$ whenever $i\neq j$. 
Because we restrict ourselves to boundedly finite measures, any bounded subset $A\subset X$ contains finitely many points of the family  $\{x_i\}_{i}$. Conversely, any countable family of points satisfying this property defines a measure $\xi\in X^*$ by the above formula. 
We define $\A^*$ as the restriction to $X^*$ of $\widetilde{\A}$.

In the paper, we consider a boundedly finite measure $\mu$ on $X$ and an invertible transformation $T$ on $X$ preserving $\mu$. 
We assume that $\mu(X)=\infty$ and that $(X,\A,\mu,T)$ is  conservative and ergodic. This implies in particular that $\mu$ is continuous.
%\todo{est-ce que tout est valable si $\mu$ est finie ?}

Given the map $T$, for any $\sigma$-finite measure $\xi$, we define $T_*(\xi)$ as the pushforward measure of $\xi$ by $T$. 
In particular, if $\xi=\sum_{i\in I}\delta_{x_{i}}$,
\[
T_{*}\left(\xi\right)=\sum_{i\in I}\delta_{T\left(x_{i}\right)}.
\]
Observe that, even if $\xi\in {X^*}$, $T_*(\xi)$ is not necessarily in $X^*$ (the property of bounded finiteness may be lost by the action of $T$). 
Nevertheless, one can consider  the smaller space
$\bigcap_{n\in\ZZ}T_*^{-n}X^*$, on which $T_*$ is a bijective transformation.
If $m$ is a probability measure on $X^*$ which is concentrated on this smaller space, then it makes sense to say that $m$ is invariant by $T_*$.
If $m$ is such a $T_*$-invariant probability measure, then $T_*(\xi)\in X^*$ for $m$-almost all $\xi\in X^*$, 
and $(X^*,\A^*,m,T_*)$ is an invertible, probability preserving dynamical system. (The same remark holds if we replace $X^*$ by $\widetilde{X}$.)

We call \emph{point process on $X$} any random variable $N$ defined on some probability space $\left(\Omega,\F ,\mathbb{P}\right)$ taking values in $\left(X^{*},{\A}^{*}\right)$. 
In this case, for $\omega\in\Omega$, $N(\omega)$ can be viewed as a (random) set of points in $X$, and  ``$x\in N(\omega)$" means $N(\omega)(\{x\})=1$.  
As usual in probability theory, we will often 
omit $\omega$ in the formulas.

\begin{definition}
Let $\left(\Omega,\F ,\mathbb{P}\right)$ be a probability space, endowed with a measure preserving invertible transformation $S$.
A \emph{$T$-point process} defined on $\left(\Omega,\F ,\mathbb{P},S\right)$ is a point process $N: \Omega\to X^*$, 
such that
\begin{itemize}
\item for any set $A\in{\A}$, $N(\omega)\left(A\right)=0$ for $\mathbb{P}$-almost all $\omega$
whenever $\mbox{\ensuremath{\mu}}\left(A\right)=0$;
\item    for $\mathbb{P}$-almost all $\omega$,  for any set $A\in{\A}$,
$N(S\omega)\left(A\right)=N(\omega)\left(T^{-1}A\right)$.
\end{itemize}
\end{definition}

Thus, a $T$-point process $N$ implements a factor relationship between the dynamical systems $\left(\Omega,\F ,\mathbb{P},S\right)$
and $\left(X^*,{\A}^{*},m,T_{*}\right)$ where $m$
is the pushforward measure of $\,\mathbb{P}$ by $N$.

Observe that the formula $A\in{\A}\mapsto\mathbb{E}\left[N\left(A\right)\right]$
defines a $T$-invariant measure which is absolutely continuous with
respect to $\mu$. It is called the \emph{intensity} of $N$ and as
soon as it is $\sigma$-finite, by ergodicity of $\left(X,{\A},\mu,T\right)$, it is a multiple of $\mu$:
\[
\mathbb{E}\left[N\left(\cdot\right)\right]=\alpha\mu\left(\cdot\right)
\]
for some $\alpha>0$. In this  case, we will say that $N$ is \emph{integrable}.
More generally, setting 
\[{\A}_{f}:=\left\{ A\in{\A},\;0<\mu\left(A\right)<+\infty\right\},\] 
we have:

\begin{definition}
A $T$-point process $N$ on $\left(\Omega,\F ,\mathbb{P},S\right)$
is said to have \emph{moments of order $n\ge1$} if, for  all  $A\in{\A}_f$, $\mathbb{E}\left[\left(N\left(A\right)\right)^{n}\right]<+\infty$.
In this case, for $k\le n$, the formula
\[
M_{k}^{N}(A_1\times\cdots\times A_k) \egdef \mathbb{E}\left[N\left(A_{1}\right)\times\cdots\times N\left(A_{k}\right)\right]
\]
defines a boundedly finite $T\times\cdots\times T$-invariant measure
$M_{k}^{N}$ on $\left(X^{k},{\A}^{\otimes k}\right)$ called
the \emph{$k$-order moment measure}.

A $T$-point process with moments of order $2$ is said to be \emph{square integrable}.
\end{definition}

\subsection{Poisson point process and SuShis}

The most important $T$-point processes are Poisson point processes, let
us recall their definition.
\begin{definition}\label{def:Poisson}
A random variable $N$ with values in $(X^*, \A^*)$ is a \emph{Poisson point process of intensity $\mu$} if 
for any $k\ge1$, for any mutually disjoint sets $A_{1},\dots,A_{k}\in{\A}_{f}$,
the random variables $N\left(A_{1}\right),\dots,N\left(A_{k}\right)$
are independent and Poisson distributed with respective parameters
$\mu\left(A_{1}\right),\dots,\mu\left(A_{k}\right)$. 
\end{definition}

Such a process always exists, and its distribution $\mu^{*}$ on $X^*$ is uniquely determined by the preceding conditions.
Since $T$ preserves $\mu$, one easily checks that $T_{*}$ preserves~$\mu^{*}$.

\begin{definition}
The probability-preserving dynamical system $\left(X^{*},{\A}^{*},\mu^{*},T_{*}\right)$ is called
the \emph{Poisson suspension} over the base $\left(X,{\A},\mu,T\right)$.
\end{definition}

The basic result (see \textit{e.g.}~\cite{Roy2007}) about Poisson suspensions states that
$\left(X^{*},{\A}^{*},\mu^{*},T_{*}\right)$ is ergodic (and
then weakly mixing) if and only if there is no $T$-invariant
set in ${\A}_{f}$. In particular this is the case if $\left(X,{\A},\mu,T\right)$
is ergodic and $\mu$ infinite.

Defining $N$ on the probability space $\left(X^{*},{\A}^{*},\mu^{*}\right)$ as the identity map provides an example of a $T$-point process, the underlying measure-preserving transformation being $S=T_*$ in this case.

More generally, we can define a class of $T$-point processes which are constructed from independent Poisson point processes.

\begin{definition}
Let $(N_i)_{i\in I}$ be a countable family of independent Poisson $T$-point processes of respective intensities $\alpha_i\mu$, defined on the same probability space $(\Omega,\F,\PP)$. For each $i\in I$, consider a finite subset $J_i\subset \ZZ$, with 
\[
  \sum_{i\in I} \alpha_i \#J_i <\infty.
\]
Then, the process $N$ defined by
\[
  N(\omega)\egdef \sum_{i\in I}\sum_{j\in J_i} T_*^j (N_i(\omega))
\]
is a particular integrable $T$-point process called \emph{Superposition of Shifted Poisson Processes}, which we abbreviate in \emph{SuShi}.
\end{definition}

Note that the aperiodicity of $T$ ensures that the realizations $N(\omega)$ of the SuShi are  indeed in $X^*$ (the supports of the $T_*^j (N_i(\omega))$ are pairwise disjoint). 

It is easy to see that a SuShi always admits a canonical decomposition, in which the subsets $J_i$ are distinct, contained in 
$\ZZ_+$, and $0\in J_i$ for all $i$: indeed, consider $m_i\egdef \min J_i$, and if $m_i\neq 0$, replace $N_i$ by $T_*^{m_i} N_i$ and $J_i$ by $J_i-m_i$. Then join together all Poisson point processes $N_i$ sharing the same finite set $J_i$.

In general, a SuShi does not have moments of all orders (with  arguments similar to those used at the end of the proof of Proposition~\ref{prop:Bout d'orbite infini}, one can show that, if $\sum_{i\in I} \alpha_i (\#J_i)^2 =\infty$, then there exists $A\in\A_f$ whith $\EE\left[N(A)^2\right]=\infty$). However, if the numbers $\#J_i$, $i\in I$ are uniformly bounded, then moments of all orders exist for the SuShi.

\subsection{Roadmap of the paper}

The paper is organized as follows. Section~\ref{sec:2} is devoted to general results about $T$-point processes, and their behaviour regarding $T$-orbits. In Section~\ref{sec:orbits}, we consider the following question: can a $T$-point process assign infinite mass to $T$-orbits?  We show in Proposition~\ref{prop:Bout d'orbite infini} that, if a $T$-point process is square-integrable, then it almost surely gives finite mass to any $T$-orbit. We also explain how to construct a (non-square-integrable) $T$-point process which almost surely assigns mass 0 or $\infty$ to any $T$-orbit (Proposition~\ref{prop:counterexample}). In Section~\ref{sec:separating_orbits}, we describe a canonical decomposition of a $T$-point process assigning finite mass to $T$-orbits (Proposition~\ref{prop:Decomposition}), which is naturally related to the form taken by SuShis. In Section~\ref{sec:interaction}, we first provide a useful criterion to detect whether several $T$-point processes, defined on the same probability space, charge 
points in the same $T$-orbits (Lemma~\ref{lemma:graph measures}). Then we prove in Proposition~\ref{prop:PoissonFree} that any $T$-point process whose 2-order moment measure coincides with the one of a Poisson process is \emph{$T$-free}, that is to say it almost surely charges at most one point in any $T$-orbit. 

Section~\ref{sec:getting_sushis} presents the key rigidity results of the paper. At the beginning of this section, we introduce two additional properties of the infinite measure preserving dynamical system $(X,\mathcal{A},\mu,T)$, denoted by~\eqref{P1} and~\eqref{P2}. The former simply says that any direct product of a finite number of copies of this system remains ergodic, and the latter is a strong restriction on the set of $T^{\times n}$-invariant measures on $X^n$. 
As proved in~\cite{nfc}, there exists an infinite measure preserving version of the Chacon transformation satisfying these properties. Then we show that, under assumptions~\eqref{P1} and~\eqref{P2}, $T$-point processes with moments of all orders are SuShis (Theorem~\ref{thm:sushi}). 
An important step for the proof of this result is the particular case where the $T$-point process is $T$-free: in this situation we prove that it has to be a Poisson process (Theorem~\ref{theo:Free}). We also need a result 
ensuring the independence of Poisson $T$-point processes which 
do not charge the same $T$-orbits, provided by Proposition~\ref{prop:dissociation}.

Section~\ref{sec:joinings} deals with the consequences of the preceding results regarding self-joinings and factors of the Poisson suspension. We begin by recalling classical results on the $L^2$ structure of a Poisson suspension (Section~\ref{sec:L2Poisson}), and presenting the central notions of Poisson factors and Poisson joinings (Section~\ref{sec:PoissonJoinings}). Then we prove that, if $T$ satisfies~\eqref{P1} and~\eqref{P2}, any ergodic infinite self-joining of the associated Poisson suspension is realized as a factor of a universal Poisson suspension (Theorem~\ref{theo:Any-ergodic-countable}).  
In fact, the result even holds for an ergodic joining of a countable family of Poisson suspensions which are all of the form $T_*^{(\alpha)}$, where for each $\alpha>0$,  $T_*^{(\alpha)}$ denotes the Poisson suspension $(X_*, \mathcal{A}^*, (\alpha\mu)^*, T_*)$.
As a corollary, we get in Theorem~\ref{theo:PaPandPrime} that such a Poisson suspension satisfies the \emph{\PaP\ property}, which means that any ergodic self-joining of this system is Poisson. This \PaP\ property turns out to be, in the Poissonian context, the analog of the so-called \emph{GAG property} for Gaussian stationary processes (see \cite{LemParThou00Gausselfjoin}). We also present in Section~\ref{sec:PaP} some general properties of \PaP\ Poisson suspensions.

In Section~\ref{sec:lemanczyk}, we see how the suspensions $T_*^{(\alpha)}$ can be used to obtain a new kind of counterexample to the famous question of Furstenberg: if two ergodic dynamical systems are not disjoint, do they share a common factor ? All counterexamples known so far have the property that one of the two non-disjoint systems shares a common factor with a distal extension of the other one. In  Proposition~\ref{prop:Talpha_Tbeta}, we show that, if $\alpha\neq\beta$, $T_*^{(\alpha)}$ and $T_*^{(\beta)}$ are two non-disjoint systems which are prime, but neither of them is a factor of a distal extension of the other one, answering negatively a question asked by Lema\'nczyk. 

In Section~\ref{sec:disjointness}, we also present disjointness results for the Poisson suspension, still assuming~\eqref{P1} and~\eqref{P2}, and a further technical assumption on $T$ which is the existence of a measurable law of large number. (It is not clear whether, in general,~\eqref{P1} and~\eqref{P2} imply this existence, however these three properties are satisfied for the nearly finite Chacon transformation.)
Under these assumptions, we prove in particular the following key result: if a dynamical system $S$ is not disjoint from the Poisson suspension $(X_*, \mathcal{A}^*, \mu^*, T_*)$, then there exists $\alpha>0$ such that 
$T_*^{(\alpha)}$ appears as a factor of $S$ (Proposition~\ref{prop:disjointness-alpha}). Applications of this result are developed in Section~\ref{sec:disjointness_from_classical_classes}, where we prove that, under the same assumptions, the Poisson suspension is disjoint from any locally rank one dynamical system (Theorem~\ref{theo:disjointness_local_rank1}), and from any standard Gaussian dynamical system (Theorem~\ref{theo:disjointnessPoissonGauss}).

% At last, Section~\ref{sec:Non-Abelian Group actions} is a modest foray into the realm of non-abelian actions where our results take an even more radical form since the Poisson suspension is the one and only SuShi here.

\section{General results about  $T$-point processes and $T$-orbits}
\label{sec:2}

\subsection{Number of points in orbits}
\label{sec:orbits}

\begin{prop}
\label{prop:Bout d'orbite infini}
  Let $N$ be a square-integrable $T$-point process  with intensity $\mu$. Then $N$ almost surely gives finite measure to any $T$-orbit.  
\end{prop}

\begin{proof}
Let  $N$ be an integrable $T$-point process  with intensity $\mu$ defined on $(\Omega,\F,\PP,S)$.
We prove the proposition in the following form: assume that, with positive probability, there exists some $x\in N(\omega)$ with 
  \[
    N(\omega)\left(\{T^kx:\ k\in\ZZ\}\right)=\infty.
  \]
  Then, we can find $A\in \A_f$ such that  
  \[
    \EE\left[N^2(A)\right]=\infty.
  \]
The first step is to prove that
  \begin{equation}
    \label{eq:firststep}
    \forall m,M>0, \text{ there exists }A\subset X\text{ with }\mu(A)=m\text{ and }\EE\left[N^2(A)\right]\ge M \mu(A).
  \end{equation}

  The main tool for this proof will be the family of Palm probability measures ${(\PP_x)}_{x\in X}$ associated with the point process $N$. For a detailed presentation of Palm measures, we refer the reader to \cite{DaleyVereJones}, Chapter 13. Recall that for $\mu$-almost each $x\in X$, $\PP_x$ is a probability measure on $X^{*}$ which can be interpreted as the distribution of $N$ conditioned on $x\in N$. For each measurable $U\subset X^{*}$, $x\mapsto \PP_x(U)$ is measurable, and we have for each measurable $A\subset X$
  \[
    \EE\left[ \ind{U}(N) N(A) \right] = \int_A \PP_x(U) d\mu(x).
  \]
  More generally, if $g$ is a positive measurable function on $X\times X^{*}$, and if we denote by $\EE_x[\,\cdot\,]$ the expectation with respect to $\PP_x$, we have
  \begin{equation}
    \label{eq:Palm}
    \EE\left[ \int_X g(x,N) dN(x) \right]  = \int_X \EE_x[g(x,\xi)] \, d\mu(x).
  \end{equation}
We claim that, for $\mu$-almost all $x$, 
\begin{equation}
\label{eq:p}
  \PP_{Tx}=(T_*)_* \PP_x.
\end{equation}
 Indeed, as $\PP$ is $S$-invariant, for each measurable $U\subset X^{*}$ and each measurable $A\subset X$, we can write
\begin{align*}
  \int_A \PP_x(U)\, d\mu(x) &= \EE\left[ \ind{U}(N) N(A) \right]\\
  &=\EE\left[ \left(\ind{U}(N) N(A)\right)\circ S \right]\\
  &= \EE\left[ \ind{T_*^{-1}U}(N) N(T^{-1}A) \right]\\
  &= \int_{T^{-1}A} \PP_x(T_*^{-1}U) \,d\mu(x).
  % &= \int_A \PP_{Tx} (T_*^{-1}U) d\mu(x). Ici c'est faux !!!
\end{align*}
But on the other hand, using $T$-invariance of $\mu$, we also have
\[
 \int_A \PP_x(U)\, d\mu(x) = \int_{T^{-1}A} \PP_{Tx} (U) \, d\mu(x). 
\]
This yields, for each measurable $U\subset X^{*}$,  $\PP_{Tx} (U) = \PP_x(T_*^{-1}U)$,  and \eqref{eq:p} follows.

Now, for each $x\in X$, let us introduce the map $\varphi_x:\ X^*\to\{0,1\}^\ZZ$ defined by
\[
  \forall k\in\ZZ,\ \left(\varphi_x(\xi)\right)_k\egdef \xi\left(T^kx\right).
\]
Observe that 
\begin{equation}
  \label{eq:fi}
\varphi_{Tx}\circ T_* = \varphi_x.
\end{equation}
Define also the probability measure $\nu_x$ on $\{0,1\}^\ZZ$ as the pushforward measure of $\PP_x$ by $\varphi_x$.
By~\eqref{eq:p} and \eqref{eq:fi}, we get, for $\mu$-almost all $x\in X$
\[
\nu_{Tx} = (\varphi_{Tx})_*(\PP_{Tx}) = (\varphi_{Tx}\circ T_*)_*(\PP_x) = (\varphi_{x})_*(\PP_x) = \nu_x.
\]
By ergodicity of $T$, it follows that there exists a probability measure $\nu$ on  $\{0,1\}^\ZZ$ such that 
$\nu_x=\nu$ for $\mu$-almost all $x\in X$.

Let us consider the following measurable function $g$ defined on $X\times X^*$:
\[
  g(x,\xi)\egdef \begin{cases}
                   1 &\text{ if }\sum_{k\in\ZZ} \xi(T^kx)=\infty\\
                   0 &\text{ otherwise.}
                 \end{cases}
\]
The hypothesis made on  the point process $N$ yields
\[
  \EE\left[ \int_X g(x,N) dN(x) \right] > 0.
\]
By~\eqref{eq:Palm}, we get that for $x$ in a subset of positive measure in $X$, 
\begin{align*}
  0 &< \EE_x[g(x,\xi)] \\
  & = \PP_x \left(\sum_k \xi(T^kx)=\infty \right) \\
  & = \PP_x \left(\sum_k {\bigl(\varphi_x(\xi)\bigr)}_k=\infty \right) \\
  & = \nu\left( \left\{ \zeta\in\{0,1\}^\ZZ: \sum_{k\in\ZZ}\zeta_k = \infty \right\}\right).
\end{align*}
It follows that
\[
  \int_{\{0,1\}^\ZZ} \sum_{k\in\ZZ} \zeta_k \, d\nu(\zeta) =\infty, 
\]
and we have
\[ \int_{\{0,1\}^\ZZ} \sum_{k\in\ZZ_+} \zeta_k \, d\nu(\zeta) =\infty,
  \text{ or }\int_{\{0,1\}^\ZZ} \sum_{k\in\ZZ_-} \zeta_k \, d\nu(\zeta) =\infty.
\]
Assume without loss of generality that the former case occurs. Then, for a given $M>0$, there exists a large integer $k_M$ such that 
\[
  \int_{\{0,1\}^\ZZ} \sum_{0\le k\le k_M} \zeta_k \, d\nu(\zeta) \ge 2M.
\]
By Rokhlin's tower theorem (\cite{Aaronson}, Theorem~1.5.9), for any $m>0$, there exists $B\subset X$, $\mu(B)=\frac{m}{2k_M+1}$, such that the sets $T^kB$, $-k_M\le k\le k_M$ are pairwise disjoint. Now set $A\egdef\sqcup_{-k_M\le k\le k_M}T^kB$, so that $\mu(A)=m$, and $A_-\egdef\sqcup_{-k_M\le k\le 0}T^kB$. Applying~\eqref{eq:Palm} with $g(x,\xi)=\ind{A}(x)\xi(A)$, we get
\[
  \EE \left[N^2(A)\right] = \int_A \EE_x [\xi(A)]\, d\mu(x) \ge  \int_{A_-} \EE_x [\xi(A)]\, d\mu(x).
\]
But, for $x\in A_-$, $T^k(x)\in A$ for each $0\le k\le k_M$, therefore
\[
  \EE_x [\xi(A)] \ge \EE_x \left[ \sum_{0\le k\le k_M} \xi(T^kx) \right] = \int_{\{0,1\}^\ZZ} \sum_{0\le k\le k_M} \zeta_k \, d\nu(\zeta) \ge 2M.
\]
Since $\mu(A_-)\ge\mu(A)/2$, we get~\eqref{eq:firststep}.

To conclude the proof of the proposition, for each $\ell\ge1$, applying~\eqref{eq:firststep}, we find $A_\ell\subset X$ satisfying $\mu(A_\ell)=1/\ell^2$ and 
$\EE[N^2(A_\ell)]\ge \ell^3\mu(A_\ell)$. Set $A\egdef \bigcup_{\ell\ge1}A_\ell$. Then $A\in\A_f$, but for each $\ell$, 
$\EE[N^2(A)]\ge\EE[N^2(A_\ell)]\ge\ell$.
\end{proof}

Without the finiteness of the second moment, we cannot conclude that the $T$-point process assigns finite mass to any $T$-orbit. Indeed, we have the following proposition. 

\begin{prop}
  \label{prop:counterexample}
  For any ergodic conservative dynamical system $(X,\A,\mu,T)$ with $\mu(X)=\infty$, there exists a $T$-point process $N$ with intensity $\mu$ such that, with probability 1, for any $x\in N(\omega)$, 
  \[
    N(\omega)\left(\{T^kx:\ k\in\ZZ\}\right)=\infty.
  \]
\end{prop}

\begin{proof}
  According to Corollary~5.3.4 in~\cite{Aaronson}, there exists a conservative ergodic dynamical system $(Y,\B,\nu,R)$ such that the direct product
  $(X\times Y,\A\otimes\B,\mu\otimes\nu,T\times R)$ is totally dissipative. Therefore, there exists a wandering set $M\subset X\times Y$ such that 
  the sets $(T\times R)^iM$, $i\in\ZZ$, are pairwise disjoint, and such that
  $X\times Y = \bigsqcup_{i\in\ZZ}(T\times R)^iM$. 
  
  Let $\eta_0$ be a Poisson process on $M$ with intensity $\mu\otimes\nu|_M$, and consider, for any $i\in \ZZ$, $\eta_i\egdef (T\times R)_*^i \eta_0$: this is a Poisson process on $(T\times R)^iM$ with intensity
  $\mu\otimes\nu|_{(T\times R)^iM}$.
  Then, set $\eta\egdef\sum_{i\in\ZZ}\eta_i$, which is a point process on $X\times Y$ with intensity $\mu\otimes\nu$. We claim now that 
  \begin{equation}
    \label{eq:invariance_by_TcroixId}
    \eta \text{ and }(T\times\Id)_*\eta\text{ have the same law.}
  \end{equation}
  Indeed, set $\tilde M\egdef (T\times\Id)M$, and partition this set into subsets $\tilde M_i\egdef\tilde M\cap (T\times R)^iM$. Then, for all $i\in\ZZ$, consider $M_i\egdef (T\times R)^{-i}\tilde M_i= (T\times R)^{-i}\tilde M\cap M$. But $T\times \Id$ commutes with $T\times R$, therefore the subsets  $(T\times R)^{-i}\tilde M$ also form a partition of $X\times Y$. It follows that the subsets $M_i$ form a partition of $M$.
  Moreover, by definition of $\eta$, 
  \[
    \eta|_{\tilde M_i}= (T\times R)^i_* \eta|_{M_i},
  \]
  and since the point processes $\eta|_{M_i}$ are independent Poisson processes, $\tilde\eta_0\egdef\eta|_{\tilde M}$ is itself a Poisson process of intensity $\mu\otimes\nu|_{\tilde M}$.
  Starting from this Poisson process defined on $\tilde M$, we can in the same way construct the point process 
  \[\tilde\eta\egdef\sum_{i\in\ZZ}(T\times R)_*^i \tilde\eta_0, \]
  which has the same distribution as $(T\times \Id)_*\eta$. But on the other hand, we have $\tilde\eta=\eta$, because these two point processes coincide on $\tilde M$ and both charge full orbits of $T\times R$. This proves~\eqref{eq:invariance_by_TcroixId}. 
  
  Finally, let us fix a measurable subset $B\subset Y$ with $\nu(B)=1$. Replacing if necessary $B$ by 
  $B\cap\{y\in Y: R^ny\in B\text{ for infinitely many integers }n\}$, which is still of measure~1, we can assume that any $y\in B$ returns infinitely often in $B$. Consider the point process on $X$
  defined by 
  \[
    \xi\egdef (\pi_X)_* \left(\eta|_{X\times B}\right),
  \]
  where $\pi_X:X\times Y\to X$ stands for the projection on the $X$ coordinate. Then, the intensity of $\xi$ is $\mu$, and by~\eqref{eq:invariance_by_TcroixId}, $\xi$ and $T_*\xi$ have the same law. Now, for any $x\in\xi$, there exists $y\in B$ such that $(x,y)\in \eta\cap X\times B$. Then there exist infinitely many integers $n$ such that $R^ny\in B$, hence such that $(T^nx,R^ny)\in \eta\cap X\times B$, and then $T^nx\in \xi$. We get the announced $T$-point process $N$ by considering $N\egdef\Id$ on $\Omega\egdef X^*$, equipped with the probability measure $\PP$ defined as the law of $\xi$, which is invariant by $S\egdef T_*$.
\end{proof}

\subsection{Separating orbits}
\label{sec:separating_orbits}
The next definition deals with the interactions between $T$-point
processes.
\begin{definition}
Two $T$-point processes $N_{1}$ and $N_{2}$ defined on $\left(\Omega,\F ,\mathbb{P},S\right)$
are said to be \emph{($T$-)dissociated} if, for $\PP$-almost all $\omega$, for all $k\in\mathbb{Z}$,
$N_{1}(\omega) \cap N_{2}\left(S^{k}\omega\right)= \emptyset$.
\end{definition}
Of course, a $T$-point process is never dissociated with itself,
however we have the following situation:
\begin{definition}
A $T$-point process $N$ is called \emph{($T$-)free} 
if for $\PP$-almost all $\omega$, for all $k\in\mathbb{Z}^{*}$,
$N(\omega) \cap N\left(S^{k}\omega\right)= \emptyset$.
\end{definition}

\begin{prop}
\label{prop:Decomposition}Let $N$ be a $T$-point
process on $\left(\Omega,\F ,\mathbb{P},S\right)$, which almost-surely assigns finite mass to any $T$-orbit. Then there
exist a finite or countable set $I$, a family $\left\{ F_{i}\right\} _{i\in I}$ of finite subsets of $\mathbb{Z_+}$, 
and a family  $\left(N_{F_{i}}\right)_{i\in I}$ of  free $T$-point processes,
measurable with respect to $N$ and mutually dissociated, such that
\[
N=\sum_{i\in I}\left(\sum_{k\in F_{i}}N_{F_{i}}\circ S^{k}\right),\;\mathbb{P}\text{-a.s.}
\]
\end{prop}
\begin{proof}
For each non-empty subset $F\subset\mathbb{Z}_+$ that contains $0$,
we can form from $N$ a $T$-point process $N_{F}$ by keeping, for
all $\omega\in\Omega$, points $x\in N\left(\omega\right)$ such that
$T^{k}x\in N\left(\omega\right)$ for all $k\in F$ and $T^{k}x\notin N\left(\omega\right)$
for all $k\in\ZZ\setminus F$. By construction, $N_{F}$ is a free $T$-point
process, and $N_{F}$ and $N_{F^{\prime}}$ are dissociated whenever $F\neq F'$. Moreover, by hypothesis, 
\[
N=\sum_{F\subset\mathbb{Z}_+,\:0<\#F<\infty,\,0\in F}N_{F},\;\mathbb{P}\text{-a.s.}
\]
Removing all sets $F$ such that $N_{F}$ vanishes $\mathbb{P}\text{-a.s.}$,
we obtain the announced decomposition.
\end{proof}

\subsection{Detecting interactions within $T$-point processes}
\label{sec:interaction}

We have already introduced the moment measures of a point process $N$ by considering the quantities
\[
\mathbb{E}\left[N\left(A_{1}\right)\cdots N\left(A_{n}\right)\right]
\]
for sets $A_{1},\dots,A_{n}$ in ${\A}$.

We also obtain a $T^{\times n}$-invariant measure on $X^n$ by considering possibly different $T$-point processes $N_1,\ldots,N_n$ defined on the same probability space, and setting, for $A_1,\ldots,A_n$ in $\A$
\[
M^{N_1,\ldots,N_n}(A_1\times\cdots\times A_n)\egdef \mathbb{E}\left[N_{1}\left(A_{1}\right)\cdots N_{n}\left(A_{n}\right)\right].
\]
If the point processes have moments of all orders, this measure is boundedly finite and captures some valuable information about the interactions between those processes. To illustrate this, the next lemma roughly says that if this measure contains a non trivial ``diagonal'' part, then it reflects the presence of points on some common orbit for some of the point processes involved.

\begin{lemma}
\label{lemma:graph measures}Let $N_{1},\dots,N_{n}$ be $n$ $T$-point
processes defined on the ergodic system $\left(\Omega,\F ,\mathbb{P},S\right)$, having moments of all orders.
Assume there exist a real number $c>0$, integers $2\le j\le n$, $k_{2},\dots, k_{j}$, and a $T^{\times(n-j)}$-invariant, $\sigma$-finite measure $\nu\neq0$ such that,
for any sets $A_{1},\dots,A_{n}$ in ${\A}_{f}$,
\[
\mathbb{E}\left[N_{1}\left(A_{1}\right)\cdots N_{n}\left(A_{n}\right)\right] \ge
c\mu\left(A_{1}\cap T^{-k_{2}}A_{2}\cap\cdots\cap T^{-k_{j}}A_{j}\right)\nu\left(A_{j_{+1}}\times\cdots\times A_{n}\right).%\le\mathbb{E}\left[N_{1}\left(A_{1}\right)\cdots N_{n}\left(A_{n}\right)\right]
\]
Then, for any $A\in{\A}_{f}$, 
\[
  \PP\left(A\cap N_1\cap T^{-k_2}N_2\cap\cdots\cap T^{-k_j} N_j\neq \emptyset\right) > 0.
\]
\end{lemma}

\begin{proof}
We can apply the ergodic theorem:
\begin{multline*}
 \frac{1}{N}\sum_{k=1}^{N}\mathbb{E}\left[N_{1}\left(A_{1}\right)\cdots N_{j}\left(A_{j}\right)\bigl(N_{j+1}\left(A_{j+1}\right)\cdots N_{n}\left(A_{n}\right)\bigr)\circ S^{k}\right]\\
 \tend{N}{\infty}\mathbb{E}\left[N_{1}\left(A_{1}\right)\cdots N_{j}\left(A_{j}\right)\right]\,\mathbb{E}\left[N_{j+1}\left(A_{j+1}\right)\cdots N_{n}\left(A_{n}\right)\right].
\end{multline*}
But 
\begin{multline*}
  \mathbb{E}\left[N_{1}\left(A_{1}\right)\cdots N_{j}\left(A_{j}\right)\bigl(N_{j+1}\left(A_{j+1}\right)\cdots N_{n}\left(A_{n}\right)\bigr)\circ S^{k}\right]\\
  = \mathbb{E}\left[N_{1}\left(A_{1}\right)\cdots N_{j}\left(A_{j}\right)N_{j+1}\left(T^{-k}A_{j+1}\right)\cdots N_{n}\left(T^{-k}A_{n}\right)\right],
\end{multline*}
therefore, as $\nu$ is $T^{\times(n-j)}$-invariant,
\begin{multline*}
 \mathbb{E}\left[N_{1}\left(A_{1}\right)\cdots N_{j}\left(A_{j}\right)\right]\mathbb{E}\left[N_{j+1}\left(A_{j+1}\right)\cdots N_{n}\left(A_{n}\right)\right]\\
  \ge  c\mu\left(A_{1}\cap T^{-k_{2}}A_{2}\cap\cdots\cap T^{-k_j}A_{j}\right)\nu\left(A_{j_{+1}}\times\cdots\times A_{n}\right).
\end{multline*}
We claim that there exists some set $B\in{\A}_{f}$ so that $\mathbb{E}\left[N_{j+1}\left(B\right)\cdots N_{n}\left(B\right)\right]>0$ and $\nu\left(B\times\cdots\times B\right)>0$.
Indeed, we first observe that, by ergodicity, we have $N_k(X)=\infty$ a.s. for all $k$. Thus, $N_{j+1}\left(X\right)\cdots N_{n}\left(X\right)=\infty$ a.s.
Take some increasing sequence ${(B_\ell)}_{\ell\ge1}$ in $\A_f$ such that $X= \bigcup_{\ell\ge1}B_\ell$.
Then 
\[
  \mathbb{E}\left[N_{j+1}\left(B_\ell\right)\cdots N_{n}\left(B_\ell\right)\right] \tend{\ell}{\infty} \infty,
\]
thus $\mathbb{E}\left[N_{j+1}\left(B_\ell\right)\cdots N_{n}\left(B_\ell\right)\right]>0$ for all large enough $\ell$. 
By the same argument, we also have $\nu(B_\ell\times\cdots B_\ell)>0$ for all large enough $\ell$, and we can take $B=B_\ell$ for some large $\ell$.

We now set 
\[\alpha:=\frac{\nu\left(B\times\cdots\times B\right)}{\mathbb{E}\left[N_{j+1}\left(B\right)\cdots N_{n}\left(B\right)\right]}>0.
  \]
Then, we have for any $A_{1},\dots,A_{j}$ in ${\A}_{f}$,
\begin{equation}
  \label{eq:alpha}
  \mathbb{E}\left[N_{1}\left(A_{1}\right)\cdots N_{j}\left(A_{j}\right)\right]\ge c\alpha\mu\left(A_{1}\cap T^{-k_{2}}A_{2}\cap\cdots\cap T^{-j}A_{j}\right).
\end{equation}
Let us consider now  a generating sequence $\left(\left(A_{_{i}}^{n}\right)_{1\le i\le p_n}\right)_{n\ge1}$
of partitions of $A$: this means that this sequence of partitions of $A$ is increasing, and that for any $x\neq y$ in $A$, there exists $n(x,y)$ such that, for any $n\ge n(x,y)$, $x$ and $y$ do not
belong to the same atom of the partition $\left(A_{_{i}}^{n}\right)_{1\le i\le p_n}$.
Observe that
\begin{multline*}
 \sum_{i=1}^{p_n}N_{1}\left(A_{_{i}}^{n}\right)N_{2}\left(T^{k_{2}}A_{_{i}}^{n}\right)\cdots N_{j}\left(T^{k_{j}}A_{_{i}}^{n}\right)\\
 \tend{n}{\infty}\#\left\{ x\in N_{1}\cap A:\, T^{k_{2}}x\in N_{2},\dots,T^{k_{j}}x\in N_{j}\right\}.
\end{multline*}
Moreover,
\begin{align*}
 &\sum_{i=1}^{p_n}N_{1}\left(A_{_{i}}^{n}\right)N_{2}\left(T^{k_{2}}A_{_{i}}^{n}\right)\cdots N_{j}\left(T^{k_{j}}A_{_{i}}^{n}\right)\\
 & \le  N_{1}\left(A\right)N_{2}\left(T^{k_{2}}A\right)\cdots N_{j-1}\left(T^{k_{j-1}}A\right)\sum_{i=1}^{p_n}N_{j}\left(T^{k_{j}}A_{_{i}}^{n}\right)\\
 & =  N_{1}\left(A\right)N_{2}\left(T^{k_{2}}A\right)\cdots N_{j-1}\left(T^{k_{j-1}}A\right)N_{j}\left(T^{k_{j}}A\right),
\end{align*}
which is integrable. So we can apply the dominated convergence theorem to get
\begin{multline*}
 \mathbb{E}\left[\sum_{i=1}^{p_n}N_{1}\left(A_{_{i}}^{n}\right)N_{2}\left(T^{k_{2}}A_{_{i}}^{n}\right)\cdots N_{j}\left(T^{k_{j}}A_{_{i}}^{n}\right)\right]\\
 \tend{n}{\infty} \mathbb{E}\left[\#\left\{ x\in N_{1}\cap A,\, T^{k_{2}}x\in N_{2},\dots,T^{k_{j}}x\in N_{j}\right\} \right].
\end{multline*}
On the other hand, we have
\begin{align*} &\mathbb{E}\left[\sum_{i=1}^{p_n}N_{1}\left(A_{_{i}}^{n}\right)N_{2}\left(T^{k_{2}}A_{_{i}}^{n}\right)\cdots N_{j}\left(T^{k_{j}}A_{_{i}}^{n}\right)\right]\\
 & = \sum_{i=1}^{p_n}\mathbb{E}\left[N_{1}\left(A_{_{i}}^{n}\right) N_{2}\left(T^{k_{2}}A_{_{i}}^{n}\right)\cdots N_{j}\left(T^{k_{j}}A_{_{i}}^{n}\right)\right]\\
 & \ge  \sum_{i=1}^{p_n}c\alpha\mu\left(A_{_{i}}^{n}\cap A_{_{i}}^{n}\cap\cdots\cap A_{_{i}}^{n}\right) \quad\text{by~\eqref{eq:alpha}}\\
 & =  c\alpha\sum_{i=1}^{p_n}\mu\left(A_{_{i}}^{n}\right)\\
 & =  c\alpha\mu\left(A\right)  
\end{align*}
Letting $n$ go to $\infty$, we get
\[
  \mathbb{E}\left[\#\left\{ x\in N_{1}\cap A,\, T^{k_{2}}x\in N_{2},\dots,T^{k_{j}}x\in N_{j}\right\} \right]\ge c\alpha\mu\left(A\right),
\]
which concludes the proof.
\end{proof}

In the case where $N_1=N_2=N$, the following proposition shows that some particular form for $M_2^N$ forces the $T$-point process
to be free. 
\begin{prop}
\label{prop:PoissonFree}
Let $N$ be a square integrable $T$-point process, whose second order
moment measure satisfies 
\[M_{2}^{N}\left(A_{1}\times A_{2}\right)=\mbox{\ensuremath{\mu}}\left(A_{1}\cap A_{2}\right)+\mu\left(A_{1}\right)\mbox{\ensuremath{\mu}}\left(A_{2}\right).\]
Then $N$ is a free $T$-point process. In particular the Poisson
process associated to the Poisson suspension $\left(X^{*},{\A}^{*},\mu^{*},T_{*}\right)$
is a free $T$-point process.
\end{prop}
\begin{proof}
Fix $k\neq0$. Take a set $A\in{\A}_{f}$
such that $A\cap T^{k}A=\emptyset$ (such a set always exists) and
$\left(\left(A_{_{i}}^{n}\right)_{1\le i\le p_n}\right)_{n\ge1}$ a generating sequence
of partitions of $A$, satisfying $\mu\left(A_{i}^{n}\right)=\frac{\mu\left(A\right)}{p_n}$ for all $1\le i\le p_n$.
Then
\[
\mathbb{E}\left[\sum_{i=1}^{p_n}N\left(A_{_{i}}^{n}\right)N\left(T^{k}A_{_{i}}^{n}\right)\right]\tend{n}{\infty}\mathbb{E}\left[\#\left\{ x\in N\cap A:\, T^{k}x\in N\right\} \right].
\]
But, as $A\cap T^{k}A=\emptyset$,
\begin{align*}
\mathbb{E}\left[\sum_{i=1}^{p_n}N\left(A_{_{i}}^{n}\right)N\left(T^{k}A_{_{i}}^{n}\right)\right] 
& = \sum_{i=1}^{p_n}M_2^N(A_i^n\times T^k A_i^n)\\
& = \sum_{i=1}^{p_n}\mu(A_i^n)^2=\frac{\mu\left(A\right)^2}{p_n}\tend{n}{\infty}0.
\end{align*}
Therefore $\#\left\{ x\in N\cap A:\, T^{k}x\in N\right\} =0$
a.s. But this is also true if we replace $A$ by $T^nA$ for any $n\in\ZZ$. 
As $T$ is conservative ergodic, the
set $A$ is a sweep out set, which means that $\cup_{n\in\mathbb{Z}}T^{n}A=X$
a.e. and we get $\#\left\{ x\in N,\, T^{k}x\in N\right\} =0$ a.s.
%As recalled in Section ?????, a Poisson process of intensity $\mu$ has the aformentionned second order moment measure.
\end{proof}

\section{Getting SuShis}%Fillets over infinite Chacon are Poisson}
\label{sec:getting_sushis}
For each $n\ge1$, let us denote by $\P_n$ the set of all partitions of $\{1,\ldots,n\}$. Given $\pi\in \P_n$, and a family $\kappa=(k_i)_{1\le i\le n}$ of integers, we can define a measure
$m_\pi^\kappa$ on $X^n$, by setting
\[
  m_\pi^\kappa (A_{1}\times\cdots\times A_{n})\egdef \prod_{P\in\pi} \mu\left( \bigcap_{i\in P} T^{-k_i}A_i \right).
\]
When $\kappa=(0,\ldots,0)$, we simply note $m_\pi$ instead of $m_\pi^{(0,\ldots,0)}$. 
When $\pi$ is the partition into points, $m_\pi^\kappa$ is the product measure $\mu^{\otimes n}$. When $\pi$ is the trivial partition with a single atom, $m_\pi$ corresponds to the $n$-diagonal measure, concentrated on $\Delta_n\egdef\{(x,\ldots,x)\in X^n:x\in X\}$.

Since $\mu$ is $T$-invariant, the measure $m_\pi^\kappa$ is $T^{\times n}$-invariant. Moreover, we can always, without changing the measure  $m_\pi^\kappa$, shift the subfamilies $(k_i)_{i\in P}$ so that 
$k_i=0$ whenever $i$ is the smallest element of the atom $P$ of $\pi$: we say in this case that $\kappa$ is \emph{$\pi$-compatible}.

The action of $T^{\times n}$ on the measure $m_\pi^\kappa$ is isomorphic to 
\[\left(X^{\#\pi},{\A}^{\otimes\#\pi},\mu^{\otimes\#\pi},T^{\times\#\pi}\right)].\]

From now on, we assume that $T$ satisfies the following properties:
% \begin{enumerate}
%   \item[(P1)] For each $n\ge1$, the product system $(X^n,\A^{\otimes n}, \mu^{\otimes n}, T^{\times n})$ is ergodic;
%   \item[(P2)] For each $n\ge1$, if $\sigma$ is a boundedly finite, $T^{\times n}$-invariant measure on $X^n$, %such that $(X^n,\sigma,T^{\times n})$ is conservative, and 
%   whose marginals are absolutely continuous with respect to $\mu$,
%   then  $\sigma$ is conservative, and its ergodic components are all of the form $m_\pi^\kappa$ for some $\pi\in\P_n$ and some $\pi$-compatible family $\kappa$. 
% \end{enumerate}
\begin{equation}
  \label{P1}\tag{P1}\parbox[t]{11cm}{For each $n\ge1$, the product system $(X^n,\A^{\otimes n}, \mu^{\otimes n}, T^{\times n})$ is ergodic;}
\end{equation}
\begin{equation}
  \label{P2}\tag{P2}\parbox[t]{11cm}{For each $n\ge1$, if $\sigma$ is a boundedly finite, $T^{\times n}$-invariant measure on $X^n$, %such that $(X^n,\sigma,T^{\times n})$ is conservative, and 
  whose marginals are absolutely continuous with respect to $\mu$,
  then  $\sigma$ is conservative, and its ergodic components are all of the form $m_\pi^\kappa$ for some $\pi\in\P_n$ and some $\pi$-compatible family $\kappa$. }
\end{equation}
An example of a transformation $T$ satisfying both properties is given by the so-called
\emph{nearly finite Chacon transformation}~\cite{nfc}.
%infinite Chacon transformation, introduced by Adams, Friedman and Silva in 1997~\cite{AFS1997}:  Property~\eqref{P1} is proved in their article, whereas we prove in~\cite{chaconinfinite} that it satisfies Property~\eqref{P2}.

\subsection{Free implies Poisson}

%The following lemma is surely well known from probabilists, we give a proof for completeness as we haven't found it in the form we need in the literature.

It is well known that a Poisson distribution is completely determined by its moments, which follows from the fact that its moment generating function is 
analytic in a neighborhood of 0 (see \textit{e.g.} \cite[p.~86]{DasGupta}). The following lemma is a kind of generalization
of this result to the distribution of Poisson point processes, which are completely determined by their moment measures.

\begin{lemma}
  \label{lemma:poisson_moment_measures}
  Let $\mathcal{N}$ be a Poisson point process on $X$ of intensity $\mu$, and assume that $N$ is a point process on $X$ with the same moment measures as $\mathcal{N}$.
  Then $N$ is also a Poisson point process of intensity $\mu$.
\end{lemma}

\begin{proof}

For any $A\in {\A}_{f}$, $N\left(A\right)$ is a Poisson random variable of intensity $\mu\left(A\right)$ and as such, its distribution is completely determined by its moments.
% since its Laplace transform is defined on $\RR$ and analytical. 
Since, by hypothesis, $\mathcal{N}\left(A\right)$ has the same moments, it is also a Poisson random variable of parameter $\mu\left(A\right)$. We conclude by applying R\'enyi's characterization theorem~\cite{Renyi1967} which, in particular, identifies as a Poisson process on $X$ of intensity $\mu$, any point process such that, for any $A\in {\A}_{f}$, the random measure of $A$ is Poisson distributed with parameter $\mu\left(A\right)$. (R\'enyi's orginal statement was restricted to Poisson processes on the real line with a non atomic Radon measure but it can of course be translated on any metric space with a continuous Borel measure.)

%The distribution of $\mathcal{N}$ is completely determined by the
%quantities 
%\[
%\mathbb{E}\left[\exp\left(\alpha_{1}\mathcal{N}\left(A_{1}\right)+\dots+\alpha_{k}\mathcal{N}\left(A_{k}\right)\right)\right]
%\]
% for any $k\ge1$ and any collection $A_{1},\dots,A_{k}$ of disjoint
%sets in ${\A}_{f}$ and positive numbers $\alpha_{1},\dots,\alpha_{k}$.
%
%Let $Y:=\alpha_{1}\mathcal{N}\left(A_{1}\right)+\dots+\alpha_{k}\mathcal{N}\left(A_{k}\right)$ be such a random variable, and set
%$Z:=\alpha_{1}{N}\left(A_{1}\right)+\dots+\alpha_{k}{N}\left(A_{k}\right)$. Since the moment measures of $N$ and $\mathcal{N}$ coincide, we 
%have for any $n\ge1$,
%\begin{equation}
%  \label{eq:YZ}
%  \mathbb{E}\left[Y^{n}\right]=\mathbb{E}\left[Z^{n}\right].
%\end{equation}
%
%The Laplace transform $\varphi\left(t\right)=\mathbb{E}[e^{-tY}]$
%is analytic on $\mathbb{R}$, which implies that the series
%\[
%\sum_{n=0}^{+\infty}\frac{\left|t\right|^{n}}{n!}\mathbb{E}\left[Y^{n}\right]
%\]
%converges (to $\varphi\left(-\left|t\right|\right)$. But this, in
%turn, implies with~\eqref{eq:YZ} that the series
%\[
%\sum_{n=0}^{+\infty}\frac{\left(-1\right)^{n}t^{n}}{n!}Z^{n}
%\]
% converges  in $L^{1}(\mathbb{P})$ to $\exp\left(-tZ\right)$.
%Therefore, the Laplace transform of $Z$ is also $\varphi\left(t\right)$
%and we obtain
%\[
%\mathbb{E}\left[\exp\left(\alpha_{1}\mathcal{N}\left(A_{1}\right)+\dots+\alpha_{k}\mathcal{N}\left(A_{k}\right)\right)\right]=\mathbb{E}\left[\exp\left(\alpha_{1}N\left(A_{1}\right)+\dots+\alpha_{k}N\left(A_{k}\right)\right)\right]
%\]
%which ends the proof.
\end{proof}

\begin{theo}
\label{theo:Free}Assume that properties~\eqref{P1} and~\eqref{P2} hold for $T$. If $N$ is a free $T$-point process with moments
of all orders defined on an ergodic system $\left(\Omega,\F ,\mathbb{P},S\right)$,
then it is Poisson.
\end{theo}

\begin{proof}
Let $N$ be a free $T$-point process with moments of all orders. We
can assume that $\mu$ is the intensity of $N$. 
In the first step of the proof, we want to show that the moment measures of any order of $N$ coincide with those of a Poisson process of intensity $\mu$ (we recall that this
latter point process is itself a free, ergodic, $T$-point process).

The $n$-order moment measure $M_{n}^{N}$ satisfies the hypothesis of Property~\eqref{P2}, hence it has at most countably many ergodic components, of the form 
$m_\pi^\kappa$ for some $\pi\in\P_n$ and some $\pi$-compatible family $\kappa$.
By lemma~\ref{lemma:graph measures} applied with $N_{1}=\cdots=N_{n}=N$, which is a free $T$-point process, we see that the contribution of any $m_\pi^\kappa$, where $\kappa\neq(0,\ldots,0)$, vanishes.
Therefore, the ergodic decomposition of  $M_{n}^{N}$ writes
\[
M_{n}^N=\sum_{\pi\in\P_{n}}c_{\pi}m_{\pi}.
\]

We first point out that for each $n\ge1$, the weight of the $n$-diagonal
measure is~$1$ (this is valid for any point process of intensity
$\mu$). Indeed, using once again a set $A\in{\A}_{f}$, and
$\left(\left(A_{i}^{\ell}\right)_{1\le i\le p_\ell}\right)_{\ell\ge1}$ a generating sequence of
partitions of $A$, we get 
\begin{multline*}
  \sum_{i=1}^{p_\ell}M_{n}^N\left(A_i^{\ell}\times\cdots\times A_i^{\ell}\right) = \mathbb{E}\left[\sum_{i=1}^{p_\ell}N\left(A_i^{\ell}\right)\cdots N\left(A_i^{\ell}\right)\right]\\
  \tend{\ell}{\infty}  
  \mathbb{E}\left[N\left(A\right)\right]=\mu\left(A\right)=\mu\left(A\cap\cdots\cap A\right).
\end{multline*}
On the other hand,
\begin{multline*}
  \sum_{i=1}^{p_\ell}M_{n}^N\left(A_i^{\ell}\times\cdots\times A_i^{\ell}\right) =M_{n}^N\left(\bigsqcup_{i=1}^{p_\ell}A_i^{\ell}\times\cdots\times A_i^{\ell}\right)\\
  \tend{\ell}{\infty} M_{n}^N\left(\Delta_{n}\cap A\times\cdots\times A\right).
\end{multline*}
%where $\Delta_{n}$ is the $n$-diagonal. 
Therefore $M_{n}^N\left(\Delta_{n}\cap A\times\cdots\times A\right)=\mu\left(A\cap\cdots\cap A\right)$, 
which implies, as claimed,  that the weight of the $n$-diagonal measure is $1$.

We now want to prove by induction that, for all $n\ge1$, $M_{n}^N$
is the $n$-order moment measure of a Poisson process of intensity $\mu$.
The property is of course satisfied for $n=1$. Let us assume it is satisfied
up to some $n\ge1$, and let $A_{1},\dots,A_{n+1}$ be sets in ${\A}_{f}$.
Pick a nonempty subset $K\subsetneq\left\{ 1,\dots,n+1\right\}$.
By the ergodic theorem, we get
\begin{multline}
\label{eq:1}
 \frac{1}{\ell}\sum_{k=1}^{\ell}\mathbb{E}\left[\prod_{i\in K}N\left(A_{i}\right)\left(\prod_{i\in K^{c}}N\left(A_{i}\right)\circ S^{k}\right)\right]\\
 \tend{\ell}{\infty}  \mathbb{E}\left[\prod_{i\in K}N\left(A_{i}\right)\right]\mathbb{E}\left[\prod_{i\in K^{c}}N\left(A_{i}\right)\right]\\
  =  M^N_{\# K }\left(\prod_{i\in K}A_{i}\right)M^N_{(n+1-\# K)}\left(\prod_{i\in K^{c}}A_{i}\right).
\end{multline}
On the other hand,
\begin{align}
  \label{eq:2}& \frac{1}{\ell}\sum_{k=1}^{\ell}\mathbb{E}\left[\prod_{i\in K}N\left(A_{i}\right)\left(\prod_{i\in K^{c}}N\left(A_{i}\right)\circ S^{k}\right)\right]\\
  \nonumber = & \frac{1}{\ell}\sum_{k=1}^{\ell}M^N_{n+1}\left(T^{-\epsilon_{k}\left(1\right)}A_{1}\times\cdots\times T^{-\epsilon_{k}\left(n\right)}A_{n}\right)\\
  \nonumber = & \sum_{\pi\in\P_{n+1}}c_{\pi}\frac{1}{\ell}\sum_{k=1}^{\ell}m_{\pi}\left(T^{-\epsilon_{k}\left(1\right)}A_{1}\times\cdots\times T^{-\epsilon_{k}\left(n+1\right)}A_{n+1}\right)
\end{align}
where $\epsilon_{k}\left(i\right)\egdef k$ if $i\in K$, and $\epsilon_{k}\left(i\right)\egdef 0$ otherwise. Coming back to the definition of $m_\pi$, we write
\[
  m_\pi\left(T^{-\epsilon_{k}\left(1\right)}A_{1}\times\cdots\times T^{-\epsilon_{k}\left(n+1\right)}A_{n+1}\right)
  = \prod_{P\in\pi} \mu\left( \bigcap_{i\in P} T^{-\epsilon_k(i)} A_i\right).
\]
Observe that, if $K$ is a union of atoms of $\pi$, we have for any $1\le k\le \ell$
\[
  m_\pi \left(T^{-\epsilon_k\left(1\right)}A_{1}\times\cdots\times T^{-\epsilon_k\left(n+1\right)}A_{n+1}\right)=m_{\pi}\left(A_{1}\times\cdots\times A_{n+1}\right).
\]
Otherwise, there exists an atom $P\in\pi$ containing indices $i\in K$ and $j\notin K$, hence with $\epsilon_{k}\left(i\right)=k$ and $\epsilon_{k}(j)=0$. We get that for some constant $C$,
\[
  m_\pi\left(T^{-\epsilon_{k}\left(1\right)}A_{1}\times\cdots\times T^{-\epsilon_{k}\left(n+1\right)}A_{n+1}\right) \le C \mu(A_j\cap T^{-k}A_i).
\]
But, since $T$ is an ergodic infinite-measure-preserving map, 
\[
  \frac{1}{\ell}\sum_{k=1}^{\ell}\mu\left(A_j\cap T^{-k}A_i\right) \tend{\ell}{\infty}0.
\]
Defining $\P_{n+1}^{K}$ as the set of partitions
$\pi\in\P_{n+1}$ where $K$ is a union of atoms
of $\pi$, the above proves that the contribution of all partitions $\pi\in\P_{n+1}\setminus\P_{n+1}^{K}$ vanishes, and we get, using~\eqref{eq:1} and \eqref{eq:2},
\[
M^N_{\# K }\left(\prod_{i\in K}A_{i}\right)M^N_{(n+1-\# K)}\left(\prod_{i\in K^{c}}A_{i}\right)=\sum_{\pi\in\P_{n+1}^K} c_{\pi}m_{\pi}\left(A_{1}\times\cdots\times A_{n+1}\right).
\]
Since $\emptyset\neq K\subsetneq\left\{ 1,\dots,n+1\right\}$, the ergodic decompositions of $M^N_{\# K}$ and $M^N_{(n+1-\# K)}$ only involve the coefficients $c_\pi$, $\pi\in\P_1\cup\cdots\cup\P_n$. Identifying the ergodic decompositions on both sides of the above equality, we see  that all
the coefficients $c_{\pi}$, $\pi\in\P_{n+1}^{K}$ are
completely determined by coefficients corresponding to partitions in $\P_1\cup\cdots\cup\P_n$. Moreover, the above argument is valid
in particular when $N$ is the Poisson process  of intensity $\mu$ (which is free by Proposition~\ref{prop:PoissonFree}). By letting
$K$ run through all strict subsets of $\left\{ 1,\dots,n+1\right\}$, and using the induction hypothesis,
we identify all but one coefficients of the ergodic decomposition
of $M^N_{n+1}$ as those of the Poisson point process of intensity
$\mu$. The only coefficient that cannot be determined by this method
is the one associated to the trivial partition of $\{1,\ldots,n+1\}$ into a single atom. But this corresponds  to the $(n+1)$-diagonal measure, and we already
know that this coefficient is $1$.
Thus we have proved the moment measures of any order of $N$ are those of a Poisson point process of intensity $\mu$, and we conclude by Lemma~\ref{lemma:poisson_moment_measures}.
\end{proof}

\subsection{Dissociation implies independence}%Independence of fillets}

\begin{prop}
\label{prop:dissociation}Assume that properties~\eqref{P1} and~\eqref{P2} hold for $T$. Let $N_{1},\dots,N_{k}$ be 
Poisson $T$-point processes, of respective intensity $\alpha_{1}\mu,\ldots,\alpha_{k}\mu$,
defined on the same ergodic system  $\left(\Omega,\F ,\mathbb{P},S\right)$. 
If these processes are mutually dissociated, then they are independent.
\end{prop}
\begin{proof}
Let $n_{1},\dots,n_{k}$ be positive numbers, $n:=n_{1}+\dots+n_{k}$,  and let $\{Q_1,\ldots,Q_k\}$ be the partition of $\{1,\ldots,n\}$ in subsets of consecutive integers of respective size $n_1,\ldots,n_k$. For any $\left\{ A_{i}\right\} _{1\le i\le n}$ in ${\A}_{f}$, set
\begin{equation}
  \label{eq:sigma}\sigma(A_1\times\cdots\times A_n)\egdef\mathbb{E}\left[\prod_{j=1}^k\prod_{i\in Q_j} N_j(A_i)\right].
\end{equation}
This defines a $T^{\times n}$-invariant measure $\sigma$ on $\left(X^{n},{\A}^{\otimes n}\right)$, which satisfies the hypotheses of Property~\eqref{P2}. Hence $\sigma$ has at most countably many ergodic components, of the form 
$m_\pi^\kappa$ for some $\pi\in\P_n$ and some $\pi$-compatible family $\kappa$.
By lemma~\ref{lemma:graph measures}, as the processes $N_{1},\dots,N_{k}$ are mutually dissociated, only partitions $\pi$  refining the partition $\{Q_1,\ldots,Q_k\}$ may appear in the ergodic decomposition of $\sigma$. 
Therefore, any ergodic component $m_\pi^\kappa$ of $\sigma$ has the form
\[
  m_\pi^\kappa(A_1\times\cdots\times A_n)= \prod_{j=1}^k \nu_j\left(\prod_{i\in Q_j}A_i\right),
\]
where each $\nu_j$ is a $T^{\times n_j}$-invariant measure. In particular, any ergodic component of $\sigma$ is invariant by the transformation
$(x_1,\ldots,x_n)\mapsto(y_1,\ldots,y_n)$, where $y_i\egdef Tx_i$ if $i\in Q_k$, and $y_i\egdef x_i$ otherwise. It follows that $\sigma$ itself is invariant by this transformation, hence the expression 
defining $\sigma(A_1\times\cdots\times A_n)$ on the right-hand side of~\eqref{eq:sigma}
is unchanged if we replace $N_k(A_i)$ by $N_k(T^{-1}A_i)$ for all $i\in Q_k$ simultaneously.
%Moreover, as Poisson processes are free $T$-point processes by Proposition~{prop:PoissonFree}, only  
Therefore, we can write for any $\left\{ A_{i}\right\} _{1\le i\le n}$ in ${\A}_{f}$ and any $L\ge1$
\begin{align*}
  \mathbb{E}\left[\prod_{j=1}^k\prod_{i\in Q_j} N_j(A_i)\right] & = \frac{1}{L}\sum_{1\le \ell\le L}
  \mathbb{E}\left[\left(\prod_{j=1}^{k-1}\prod_{i\in Q_j} N_j(A_i)\right) \prod_{i\in Q_k} N_k(T^{-\ell} A_i)\right] \\
  & = \mathbb{E}\left[\left(\prod_{j=1}^{k-1}\prod_{i\in Q_j} N_j(A_i)\right) 
  \left(\frac{1}{L}\sum_{1\le \ell\le L} \prod_{i\in Q_k} N_k\circ S^\ell (A_i)\right) \right] .
\end{align*}
By the ergodic theorem, this converges as $L\to\infty$ to 
\[
  \mathbb{E}\left[\prod_{j=1}^{k-1}\prod_{i\in Q_j} N_j(A_i)\right]
  \mathbb{E}\left[ \prod_{i\in Q_k} N_k (A_i)\right].
\]
A straightforward induction on $k$ then yields the equality
\[
  \mathbb{E}\left[\prod_{j=1}^k\prod_{i\in Q_j} N_j(A_i)\right]
  = \prod_{j=1}^k \mathbb{E}\left[\prod_{i\in Q_j} N_j(A_i)\right],
\]
and this is sufficient to obtain the independence between the Poisson
processes.
\end{proof}

Compiling the previous results, we now get the following structure theorem.
\begin{theo}
\label{thm:sushi}
Assume that properties~\eqref{P1} and~\eqref{P2} hold for $T$. Let $N$ be a $T$-point process  with moments of all orders defined
on an ergodic system $\left(\Omega,\F ,\mathbb{P},S\right)$.
Then $N$ is a SuShi.
\end{theo}
\begin{proof}
  Since $N$ is square integrable, Proposition~\ref{prop:Bout d'orbite infini} ensures that $N$ almost surely gives a finite measure to any $T$-orbit.
  We can therefore apply Proposition~\ref{prop:Decomposition} to write $N$ as
\[
N=\sum_{i\in I}\left(\sum_{k\in F_{i}}N_{F_{i}}\circ S^{k}\right),\;\mathbb{P}\text{-a.s.}
\]
where $I$ is countable, each $F_i$ is a finite subset of $\ZZ$, and the $T$-point processes $N_{F_i}$ are free and mutually dissociated.
Then Theorem~\ref{theo:Free} proves that each $N_{F_i}$ is a Poisson process, and Proposition~\ref{prop:dissociation} shows that they are independent.
\end{proof}

\section{Application to the structure of Poisson joinings}
\label{sec:joinings}

\subsection{Notions on the $L^{2}$ structure of a Poisson suspension}
\label{sec:L2Poisson}

There is a strong relationship between the $L^{2}$-spaces of the
suspension and the underlying space. Namely, $L^{2}\left(\mu^{*}\right)$
can be seen as the Fock space of $L^{2}\left(\mu\right)$ (see~\cite{Neretin1996}),
that is
\[
L^{2}\left(\mu^{*}\right)\simeq\text{Fock}\left(L^{2}\left(\mu\right)\right):=\mathbb{C}\oplus L^{2}\left(\mu\right)\oplus L^{2}\left(\mu\right)^{\odot2}\oplus\cdots\oplus L^{2}\left(\mu\right)^{\odot n}\oplus\cdots,
\]
where $L^{2}\left(\mu\right)^{\odot n}$ stands for the $n$-order
symmetric tensor power of $L^{2}\left(\mu\right)$, and the inner
product given on $L^{2}\left(\mu\right)^{\odot n}$ is given by
\[
\left\langle f^{\otimes n},g^{\otimes n}\right\rangle _{\text{Fock}\left(L^{2}\left(\mu\right)\right)}:=\frac{1}{n!}\left\langle f,g\right\rangle ^{n}.
\]

This means there is a sequence $\left\{ H^{n}\right\} _{n\ge0}$ of
so-called \emph{(Poissonian) chaos} which are orthogonal subspaces inside $L^{2}\left(\mu^{*}\right)$,
such that $L^{2}\left(\mu^{*}\right)=%\overline
{\bigoplus_{n\ge0}H^{n}}$,
and where, for each $n\ge1$,  $H^{n}$ is identified to $L^{2}\left(\mu\right)^{\odot n}$
through multiple integrals ($H^{0}$ corresponds to constant functions,
identified to $\mathbb{C}$). In this paper we only need to know what
happens in the first chaos: $H^{1}$ is linearly spanned by functions
$N\left(A\right)-\mu\left(A\right)$, for $A\in\mathcal{A}_{f}$,
and $N\left(A\right)-\mu\left(A\right)\in H^{1}$ corresponds to $\ind{A}\in L^{2}\left(\mu\right)$.
We have the isometry relation
\[
\Bigl\langle N\left(A\right)-\mu\left(A\right),N\left(B\right)-\mu\left(B\right)\Bigr\rangle _{L^{2}\left(\mu^{*}\right)}=\left\langle \ind{A},\ind{B}\right\rangle _{L^{2}\left(\mu\right)}.
\]

If $\varphi$ is a linear operator from  $L^{2}\left(\mu_1\right)$ to  $L^{2}\left(\mu_2\right)$,
of norm less than or equal to $1$, then $\varphi$ extends naturally to an operator
$\widetilde{\varphi}$, called the \emph{exponential of $\varphi$},
 from $\text{Fock}\left(L^{2}\left(\mu_1\right)\right)$ to  $\text{Fock}\left(L^{2}\left(\mu_2\right)\right)$, by 
\[
\widetilde{\varphi}\left(f^{\otimes n}\right):=\varphi\left(f\right)\otimes\cdots\otimes\varphi\left(f\right).
\]
In particular, if $U_{T}$ and $U_{T_{*}}$ denote the unitary operators
associated to $T$ and $T_{*}$ on their respective spaces, then,
through the identification, we have
\[
U_{T_{*}}=\widetilde{U_{T}}.
\]

\subsection{Poisson factors and Poisson joinings}
\label{sec:PoissonJoinings}
There is a distinguished collection of factors within a Poisson suspension
$\left(X^{*},{\A}^{*},\mu^{*},T_{*}\right)$ that are also Poisson suspensions:
\begin{definition}[Poisson factor]
Let $Y\subset X$ be a $T$-invariant measurable set, and let $\mathcal{C}\subset{\A}_{\mid Y}$
be a $\sigma$-finite sub-$\sigma$-algebra of ${\A}$ restricted
to $Y$. Then the Poisson suspension $\left(Y^{*},\mathcal{C}^{*},\mu^{*},T_{*}\right)$ is a natural factor of $\left(X^{*},{\A}^{*},\mu^{*},T_{*}\right)$ via the factor map
\[
  \xi\in X^{*}\longmapsto \xi|_{\mathcal{C}}\in Y^{*}.
\]
Such a factor is called a \emph{Poisson factor} of the suspension $\left(X^{*},{\A}^{*},\mu^{*},T_{*}\right)$. 
\end{definition}

Let $\left(Y_{i},\mathcal{Y}_{i},\rho_{i},R_{i}\right)$, $i\in I$, be a finite or countable family of dynamical systems. We recall that a \emph{joining} of these dynamical systems
is a measure on $\prod_iY_{i}$, invariant by the product transformation $\prod_i R_i: (y_i)_{i\in I} \mapsto (R_i y_i)_{i\in I}$, and whose marginal on each coordinate $i$ is $\rho_i$. %projects onto $\rho_{i}$ on the coordinate $i$, i.e. $m\left(Y_{1}\times Y_{i-1}\times\cdot\times Y_{i+1}\times\cdots\times Y_{n}\right)=\rho_{i}$.

Observe that this definition is not restricted to probability measure
preserving systems, but extends to the case where measures are $\sigma$-finite.
However it is worth to note that the product measure is not a joining in the infinite measure case (its marginals are not $\sigma$-finite).

To a joining $m$ of two systems $\left(Y_{i},\mathcal{Y}_{i},\rho_{i},R_{i}\right)$, $i=1,2$, corresponds a Markov operator $\varphi:L^2(\rho_1)\to L^2(\rho_2)$,
defined by
\[
  \forall A\in \mathcal{Y}_1, B\in \mathcal{Y}_2 \text{ with finite measure,}\quad m(A\times B) \egdef \int_{B} \varphi(\ind{A})\, d\rho_2.
\]

A \emph{self-joining of order $n$} is a joining of $n$ identical
copies of the same system. 

\medskip

The structure of Poisson suspensions allows one to define a natural family of joinings which plays a central role in this work.

\begin{definition}[Poisson joining]
\label{def:PoissonJoining}
 Let $\left(X_{i}^{*},{\A}_{i}^{*},\mu_{i}^{*},\left(T_{i}\right)_{*}\right)$, $i\in I$ be a finite or countable family of Poisson suspensions.
 Assume that  $(Z^*,\Z^*,m^*,R_*)$ is some other Poisson suspension, that  $(Z_i)_{i\in I}$ is a family of $R$-invariant subsets of $Z$, 
 and that for each $i\in I$ we are given a measurable map $\pi_i: Z_i\to X_i$
 such that
 \begin{itemize}
   \item $(\pi_i)_*(m|_{Z_i}) = \mu_i$,
   \item $\pi_i\circ R=T_i\circ \pi_i$.
 \end{itemize}
Let $N$ be a Poisson $R$-point process of distribution $m^*$. Then the distribution of $\Bigl((\pi_i)_*(N|_{Z_i})\Bigr)_{i\in I}$ is a joining of the Poisson suspensions 
$\left(X_{i}^{*},{\A}_{i}^{*},\mu_{i}^{*},\left(T_{i}\right)_{*}\right)$, which we call  a \emph{Poisson joining}.
\end{definition}

Let us recall the probabilistic notion of infinite divisibility, which is useful for the study of 
Poisson joinings. 
The addition of $\sigma$-finite measures on $\left(X,{\A}\right)$
is measurable and well defined and so is the convolution of distributions
on $\left(\widetilde{X},\widetilde{\A}\right)$: $m_{1}*m_{2}$ is
the pushforward measure of $m_{1}\otimes m_{2}$ by the application
\[
\begin{array}{ccc}
\left(\widetilde{X}\times\widetilde{X},\widetilde{\A}\otimes\widetilde{\A}\right) & \to & \left(\widetilde{X},\widetilde{\A}\right)\\
\left(\nu_{1},\nu_{2}\right) & \mapsto & \nu_{1}+\nu_{2}
\end{array}
\]

\begin{definition}[Infinite divisibility]
A probability measure $m$ on $\left(\widetilde{X},\widetilde{\A}\right)$
is said to be \emph{infinitely divisible} if, for every $k\ge2$, there exists
a probability measure $m_{k}$ such that
\[
m=(m_{k})^{*k}:=m_{k}*\cdots*m_{k}.
\]
\end{definition}
The distribution of a Poisson point process is easily seen to be infinitely
divisible, as we have
\[
\mu^{*}=\left(\left(\frac{1}{k}\mu\right)^*\right)^{*k},
\]
which is the formula capturing the fact that the independent superposition
of $k$ Poisson point processes of intensities $\frac{1}{k}\mu$ is
a Poisson point process of intensity $\mu$.

Observe that a pair of measures $(\xi_1,\xi_2)\in \tilde X_1\times \tilde X_2$ is naturally identified with a measure on the disjoint union $X_1\sqcup X_2$. 
Therefore, a distribution on $\tilde X_1\times \tilde X_2$ is itself identified to a distribution on $\widetilde{(X_1\sqcup X_2)}$, and we use this identification to define infinite divisibility of a joining of Poisson suspensions. 

Poisson joinings of two Poisson suspensions were defined independently using Markov operators in \cite{Lem05ELF}, 
and infinite divisibility in \cite{Roy07Infinite}, where both definitions were proved to be equivalent. Combining the results of these two papers, we get the following proposition.

\begin{prop}
\label{prop:inf-divis}
Let $\gamma$ be a joining of two Poisson suspensions. The following properties are equivalent:
\begin{enumerate}
  \item $\gamma$ is a Poisson joining.
  \item $\gamma$ is infinitely divisible.
  \item The Markov operator associated to $\gamma$ is the exponential of a sub-Markov operator defined between the
  $L^2$ spaces of the bases. 
\end{enumerate}
\end{prop}

In~\cite{Lem05ELF} and~\cite{Roy07Infinite}, Poisson joinings of two Poisson suspensions 
$\left(X_{i}^{*},{\A}_{i}^{*},\mu_{i}^{*},\left(T_{i}\right)_{*}\right)$, $i=1,2$, are in fact  characterized 
by the following structure, which is easily seen to fit our definition~\ref{def:PoissonJoining}. 
Let us first fix two measures $\gamma_{1}\le\mu_{1}$ and $\gamma_{2}\le\mu_{2}$, respectively invariant by $T_1$ and $T_2$. 
Then consider a joining $\left(X_{1}\times X_{2},{\A}_{1}\otimes{\A}_{2},m,T_{1}\times T_{2}\right)$
of $\left(X_{1},{\A}_{1},\gamma_{1},T_{1}\right)$ and
$\left(X_{2},{\A}_{2},\gamma_{2},T_{2}\right)$, and form the Poisson suspension 
\[\left(\left(X_{1}\times X_{2}\right)^{*},\left({\A}_{1}\otimes{\A}_{2}\right)^{*},m^{*},\left(T_{1}\times T_{2}\right)_{*}\right).
  \]
Now we project the points of the Poisson process on $X_{1}\times X_{2}$
of intensity $m$ on both axes $X_{1}$ and $X_{2}$,  getting two
Poisson processes with intensities $\gamma_{1}$ and $\gamma_{2}$.
This defines a factor map 
\[
\nu\in\left(X_{1}\times X_{2}\right)^{*}\mapsto\Bigl(\nu\left(\cdot\times X_{2}\right),\nu\left(X_{1}\times\cdot\right)\Bigr) \in X_1^*\times X_2^*,\]
and the factor we obtain is a joining $\left(X_{1}^{*}\times X_{2}^{*},{\A}_{1}^{*}\otimes{\A}_{2}^{*},\widetilde{m},\left(T_{1}\right)_{*}\times\left(T_{2}\right)_{*}\right)$
of the two Poisson suspensions $\left(X_{1}^{*},{\A}_{1}^{*},\gamma_{1}^{*},\left(T_{1}\right)_{*}\right)$
and $\left(X_{2}^{*},{\A}_{2}^{*},\gamma_{2}^{*},\left(T_{2}\right)_{*}\right)$.
In order to adjust intensities when $\gamma_i < \mu_i$, we superpose on each side an independent
Poisson process of intensity $\mu_i-\gamma_i$, $i=1,2$. Formally, we consider the direct product of 
the three systems $\left(X_{1}^{*},{\A}_{1}^{*},\left(\mu_{1}-\gamma_{1}\right)^{*},\left(T_{1}\right)_{*}\right)$,
$\left(X_{1}^{*}\times X_{2}^{*},{\A}_{1}^{*}\otimes{\A}_{2}^{*},\widetilde{m},\left(T_{1}\right)_{*}\times\left(T_{2}\right)_{*}\right)$
and $\left(X_{2}^{*},{\A}_{2}^{*},\left(\mu_{2}-\gamma_{2}\right)^{*},\left(T_{2}\right)_{*}\right)$,
and form 
\[
\left(X_{1}^{*}\times X_{2}^{*},{\A}_{1}^{*}\otimes{\A}_{2}^{*},\widehat{m},\left(T_{1}\right)_{*}\times\left(T_{2}\right)_{*}\right)
\]
through the factor map from $\Bigl(X_{1}^{*}\times\left(X_{1}^{*}\times X_{2}^{*}\right)\times X_{2}^{*}\Bigr)$
to $\left(X_{1}^{*}\times X_{2}^{*}\right)$ defined by
\[
\left(\widetilde{\nu}_{1},\left(\nu_{1},\nu_{2}\right),\widetilde{\nu}_{2}\right)\mapsto\left(\widetilde{\nu}_{1}+\nu_{1},\nu_{2}+\widetilde{\nu}_{2}\right).
\]
Then $\widehat{m}$ is a Poisson joining of the two Poisson suspensions $\left(X_{i}^{*},{\A}_{i}^{*},\mu_{i}^{*},\left(T_{i}\right)_{*}\right)$, $i=1,2$.

The sub-Markov operator mentioned in Proposition~\ref{prop:inf-divis} corresponds to the joining $m$ in the above description, which can also be seen as a 
  \emph{sub-joining} of $\left(X_{1},{\A}_{1},\mu_{1},T_{1}\right)$ and
$\left(X_{2},{\A}_{2},\mu_{2},T_{2}\right)$.

\begin{example}
The main situation that occurs in this paper is a Poisson self-joining of order~2
of $\left(X^{*},{\A}^{*},\mu^{*},T_{*}\right)$ where
\[
m:=\sum_{k\in\mathbb{Z}}a_{k}\Delta_{T^{k}}
\]
with $a_{k}\ge0$ and $\sum_{k\in\mathbb{Z}}a_{k}\le1$. If $a_{k}=1$
for some $k$, the corresponding Poisson
joining $\widehat{m}$ is the graph joining $\Delta_{(T_{*})^{k}}$. If $\sum_{k\in\mathbb{Z}}a_{k}=0$,
then $\widehat{m}$ is the product joining. In the other cases, we get a combination
of both.\end{example}

\subsection{Embedding joinings in a universal suspension}
\label{sec:BigMama}

\begin{lemma}
\label{lemma:multiple Poisson}Assume that properties~\eqref{P1} and~\eqref{P2} hold for $T$.
Let $\left\{ N_{i}\right\} _{i\in I}$
and $N$ be Poisson $T$-point processes defined on the ergodic
dynamical system $\left(\Omega,\F ,\mathbb{P},S\right)$, where
$I$ is at most countable and the $\left\{ N_{i}\right\} _{i\in I}$
are independent. 

Then there exists a collection of (eventually
vanishing) independent Poisson $T$-point processes $\left( N_\infty, \left\{ N_{i,k}\right\} _{i\in I,\, k\in\mathbb{Z}\cup\infty}\right)$,
 also defined on $\left(\Omega,\F ,\mathbb{P}\right)$,
measurable with respect to $\sigma\left(N,\,\left\{ N_{i}\right\} _{i\in I}\right)$,
such that
\[
N_{i} = N_{i,\infty}+\sum_{k\in\mathbb{Z}}N_{i,k},
\]
and
\[
N = N_{\infty}+\sum_{i\in I}\sum_{k\in\mathbb{Z}}N_{i,k}\circ S^{k}.
\]
\end{lemma}
\begin{proof}
For each $i$, we consider the pair $\left(N,N_{i}\right)$.
For every $k\in\mathbb{Z}$, the points $x\in N_{i}\left(\omega\right)$
such that $T^{k}x\in N\left(\omega\right)$ define a point process
$N_{i,k}$. By Proposition~\ref{prop:PoissonFree}, $N$ and the $N_i$ are free. Hence, we obtain
\[
\sum_{k\in\mathbb{Z}}N_{i,k}\le N_{i}
\]
 and
\[
\sum_{i\in I}\sum_{k\in\mathbb{Z}}N_{i,k}\circ S^{k}\le N.
\]
We define also $N_{i,\infty}:=N_{i}-\sum_{k\in\mathbb{Z}}N_{i,k}$
and $N_{\infty}:=N-\sum_{i\in I}\sum_{k\in\mathbb{Z}}N_{i,k}\circ S^{k}$.
Then the processes $\left\{ N_{i,k}\right\} _{i\in I,\, k\in\mathbb{Z}\cup\infty}$
and $N_{\infty}$ are free and mutually dissociated $T$-point processes
defined on the ergodic dynamical system $\left(\Omega,\F ,\mathbb{P},S\right)$.
They are therefore independent Poisson $T$-point processes by Theorem~\ref{theo:Free} and Proposition~\ref{prop:dissociation}.
\end{proof}

\begin{theo}
\label{theo:Any-ergodic-countable}
Assume that properties~\eqref{P1} and~\eqref{P2} hold for $T$. 
Let $(\alpha_1,\alpha_2,\ldots)$  be a finite or countable family of positive real numbers.
Then any ergodic joining of the family of Poisson suspensions $\left(X^{*},{\A}^{*},(\alpha_i\mu)^{*},T_{*}\right)$
is a Poisson joining, which can be obtained as a %Poissonian 
factor of the ergodic Poisson suspension 
$$\left(\left(X\times\mathbb{R}_{+}\right)^{*},\left({\A}\otimes\mathcal{B}\right)^{*},\left(\mu\otimes\lambda\right)^{*},\left(T\times \Id\right)_{*}\right),$$
where $\mathcal{B}$ and $\lambda$ denote the Borel $\sigma$-algebra and the Lebesgue measure on $\RR_+$ respectively.
\end{theo}

\begin{proof}
We only consider the case of a joining of a countably infinite family, the case of a finite family being covered by the same proof up to obvious changes in notations. 
Let $\left(\Omega,\F ,\mathbb{P},S\right):=\left(X^{*\mathbb{N}},\left({\A}^{*}\right)^{\otimes\mathbb{N}},m,\left(T_{*}\right)^{\otimes\mathbb{N}}\right)$
be an ergodic joining
of the countable family of Poisson suspensions $\left(X^{*},{\A}^{*},(\alpha_i\mu)^{*},T_{*}\right)$, $i\in\mathbb{N}$.
For each $j\ge 1$, we define on this space the Poisson $T$-point processes of intensity $\alpha_j\mu$
\[
N_j\left(\nu_{1},\nu_2,\ldots\right) \egdef \nu_{j}.
\]

Let $\mathcal{N}$ be the Poisson point process with intensity $\mu\otimes\lambda$
on $X\times\mathbb{R}_{+}$. Observe first that, if $J\subset\RR_+$ is an interval of length $\beta$, the random measure $\mathcal{N}\left(\cdot\times J\right)$ is a Poisson $T$-point process of intensity
$\beta\mu$ (here the underlying transformation is $(T\times \Id)_*$). Moreover, if we take disjoint subintervals of $\RR_+$, the corresponding Poisson $T$-point processes obtained in this way are independent.  In particular, $\tilde N_1\egdef \mathcal{N}\left(\cdot\times [0,\alpha_1)\right)$ has the same distribution as $N_1$. Observe also that $\tilde N_1$ can be written as $\pi_*(\mathcal{N}|_{Z_1})$, where $Z_1\egdef X\times [0,\alpha_1)$ and $\pi$ is the projection on $X$.

Set $M_1\egdef N_1$, and $\tilde M_1\egdef \tilde N_1$. Now assume that, for some $n\ge1$, we have found a finite or countable family $(M_i)_{i\in I}$ of independent Poisson $T$-point processes, measurable with respect to $(N_1,\ldots,N_n)$, such that for $1\le j\le n$,
\begin{equation}
\label{eq:decomposition_of_Nj}
  N_j = \sum_{i\in I_j} M_i\circ S^{k(i,j)},
\end{equation}
where $I_j\subset I$ and $k(i,j)\in\ZZ$. Let $\beta_i\ge0$ be such that $M_i$ has intensity $\beta_i \mu$. Assume also that we have a family $(J_i)_{i\in I}$ of disjoint subintervals of $\RR_+$ of respective length $\beta_i$. Then, the family of $T$-point processes $\tilde M_i\egdef \mathcal{N}\left(\cdot\times J_i\right)$  has the same distribution as $(M_i)_{i\in I}$, and the formula
\[
  \tilde N_j \egdef \sum_{i\in I_j} \tilde M_i\circ ((T\times \Id)_*)^{k(i,j)}
\]
yields a family $(\tilde N_1,\ldots, \tilde N_n)$ of $T$-point processes defined on \[\left(\left(X\times\mathbb{R}_{+}\right)^{*},\left({\A}\otimes\mathcal{B}\right)^{*},\left(\mu\otimes\lambda\right)^{*},\left(T\times \Id\right)_{*}\right),\]
which has the same distribution as $(N_1,\ldots,N_n)$.  Moreover, each $\tilde N_j$, $1\le j\le n$ can be written as $(\pi_j)_*(\mathcal{N}|_{Z_j})$, where $Z_j\egdef X\times \left(\bigcup_{i\in I_j} J_i\right)$, and $\pi_j$ is $T^{k(i,j)}\circ\pi$ on $X\times J_i$.

By Lemma~\ref{lemma:multiple Poisson} applied to the collection $(M_i)$ and $N_{n+1}$,
we obtain a countable family of independent
Poisson $T$-point processes $\left\{ M_{i,k}\right\} _{i\in I, k\in\mathbb{Z}\cup\{\infty\}}$ and 
$N_{n+1,\infty}$, such that for each $i\in I$, 
\[
  M_i = \sum_{k\in \ZZ\cup\{\infty\}} M_{i,k}, 
\]
and 
\[
  N_{n+1} = N_{n+1,\infty} + \sum_{i\in I}\sum_{k\in \ZZ} M_{i,k} \circ S^k.
\]
In particular, any $N_j$, $1\le j\le n+1$ can be reconstructed from the family $M_{i,k}$ and $N_{n+1,\infty}$ with a formula similar to~\eqref{eq:decomposition_of_Nj}. 
Each $M_{i,k}$ has intensity $\beta_{i,k}\mu$ for some $\beta_{i,k}\ge0$, and 
\[
  \beta_i = \sum_{k\in \ZZ\cup\{\infty\}} \beta_{i,k}. 
\]
We can therefore partition $J_i$ into disjoint subintervals $J_{i,k}$ of respective length $\beta_{i,k}$. 
Let $\beta\ge 0$ be such that $N_{n+1,\infty}$ has intensity $\beta\mu$. We can then find an extra subinterval of $\RR_+$, disjoint from $\bigcup_{i\in I}J_i$, of length $\beta$.  From this family of disjoint subintervals of $\RR_+$ we can construct a family of independent Poisson $T$-point processes $\left(\tilde M_{i,k}\right)$ and $\tilde N_{n+1,\infty}$. Then, setting
\[
  \tilde N_{n+1} \egdef \tilde N_{n+1,\infty} + \sum_{i\in I}\sum_{k\in \ZZ} \tilde M_{i,k} \circ  ((T\times \Id)_*)^k,
\]
 we get a family $(\tilde N_1,\ldots, \tilde N_n, \tilde N_{n+1})$ of $T$-point processes  which has the same distribution as $(N_1,\ldots,N_n, N_{n+1})$.
 
 By induction we get a family $(\tilde N_1,\tilde N_2,\ldots)$, defined on $$\left(\left(X\times\mathbb{R}_{+}\right)^{*},\left({\A}\otimes\mathcal{B}\right)^{*},\left(\mu\otimes\lambda\right)^{*},\left(T\times \Id\right)_{*}\right),$$ which has the same distribution as $(N_1,N_2\ldots)$. Moreover, each $\tilde N_j$ can be written as $(\pi_j)_*(\mathcal{N}|_{Z_j})$, where $Z_j$ is a $T\times\Id$-invariant subset of $X\times\RR_+$, and $\pi_j:Z_j\to X$ satisfies the requirements of Definition~\ref{def:PoissonJoining}. This ends the proof of the theorem.
 \end{proof}
 
\subsection{The \PaP\ property}
\label{sec:PaP}

In \cite{PaP}, the natural question of the existence of Poisson suspensions
with Poisson joinings as only ergodic self-joinings was addressed. This
lead to the following definition:
\begin{definition}
A Poisson suspension whose all ergodic self-joinings of order $n$
(resp. countable ergodic self-joinings) are Poisson is said to be \PaP$(n)$
(resp. \PaP$\left(\infty\right)$ (from the French
``Poisson à autocouplages Poissons''). \PaP$\left(2\right)$
will be shortened as \PaP.
\end{definition}
This notion is inspired by, and therefore closely related to, the
so-called \emph{GAG} property for Gaussian stationary processes (see \cite{LemParThou00Gausselfjoin}). Indeed
GAG Gaussian stationary processes are the processes whose ergodic
self-joinings remain Gaussian.

We first present some consequences of the \PaP\ property.

\begin{prop}
A Poisson suspension with the \PaP$\left(\infty\right)$ property has the so-called \emph{PID property} (\textit{i.e.} for any $n$, any self-joining of order $n$ with pairwise independent coordinates is the product measure).
\end{prop}
\begin{proof}
Consider an $n$-order self-joining of a \PaP$\left(\infty\right)$ suspension with pairwise independence. With the notation of Definition \ref{def:PoissonJoining}, we obtain that any two coordinates of the self-joining are associated to pairwise disjoint sets $Z_i \subset Z$. But by elementary properties of Poisson point processes recalled in Definition \ref{def:Poisson}, we obtain global independence.
\end{proof}

Recall that the \emph{centralizer} of an invertible measure preserving transformation $S$ is the set, denoted by $C(S)$, of all  invertible measure preserving transformations on the same space which commute with $S$.

\begin{prop}\label{prop:centralizer}
Let $\left(X^{*},\mathcal{A}^{*},\mu^{*},T_{*}\right)$ be a \PaP\ 
suspension and let $R\in C\left(T_{*}\right)$. Then there exists
$S\in C\left(T\right)$ such that $R=S_{*}$. In particular $C\left(T_{*}\right)\simeq C\left(T\right)$.\end{prop}
\begin{proof}
$R$ induces an ergodic self-joining of the suspension which is Poisson
thanks to the \PaP\  property. Therefore the associated
Markov operator has the form $\widetilde{\varphi}$ for some sub-Markov
operator $\varphi$ on $L^{2}\left(\mu\right)$ that commutes with
$T$. But as $\widetilde{\varphi}$ is an isometry, $\varphi$ is
also an isometry and is therefore induced by a measure preserving map
$S$ of $\left(X,\mathcal{A},\mu\right)$ that commutes with $T$,
i.e. an element of $C\left(T\right)$. By identification, $R=S_{*}$.
\end{proof}

The next proposition is very similar to Theorem~2 in~\cite{ParreauRoy}. 

\begin{prop}\label{prop:subPoissonFactor}
Let $\left(X^{*},\mathcal{A}^{*},\mu^{*},T_{*}\right)$ be a \PaP\ 
suspension and $\mathcal{K}\subset\mathcal{A}^{*}$ a non-trivial
factor. Then there exists a non-trivial Poisson factor contained in
$\mathcal{K}$, that is, there exists a $T$-invariant set of positive
measure $Y\subset X$ and a $T$-invariant $\sigma$-finite $\sigma$-algebra
$\mathcal{C}\subset\mathcal{A}_{\mid Y}$ such that $\mathcal{C}^{*}\subset\mathcal{K}$.
\end{prop}

\begin{proof}
Let $\Phi$ be the conditional expectation corresponding to $\mathcal{K}$, which is the Markov operator
associated to the relatively independent self-joining over $\mathcal{K}$.
The ergodic decomposition of this joining allows one to write 
$\Phi$  as an integral of indecomposable operators corresponding to ergodic self-joinings. 
By the \PaP\ property, these operators are 
of exponential form, \textit{i.e.} we have: 
\[
\Phi=\int_{\mathcal{W}}\widetilde{\Psi}_{w}\, \rho\left(dw\right)
\]
for some probability space $\left(\mathcal{W},\mathcal{B},\rho\right)$.
As each $\widetilde{\Psi}_{w}$ preserves the first chaos, so does
$\Phi$. Moreover, if $\Phi$ vanishes on the first chaos, so does $\widetilde{\Psi}_{w}$
for $\rho$-almost every $w\in\mathcal{W}$. But the only Markov operator
of exponential form that vanishes on the first chaos is the projection
on the constants. This means that $\Phi$ is also this projection, which corresponds
to the conditional expectation on the trivial factor, yielding a contradiction.
Thus $\Phi$ does not vanish on the first chaos and we can apply Proposition~$1$ in~\cite{ParreauRoy}: $\Phi$ induces on $L^2(\mu)$ a sub-Markov operator $\varphi$, and there exists a $T$-invariant set $Y\subset X$ such that $\varphi$ restricted to $L^2(\mu|_Y)$ is a conditional expectation on a $\sigma$-finite factor $\mathcal{C}\subset \mathcal{A}|_Y$, and $\varphi$ vanishes on $L^2(\mu|_{Y^c})$. Coming back to $L^2(\mu^*)$, $\Phi$ coincides with the exponential operator $\widetilde\varphi$ on the first chaos. Therefore its image contains all the vectors of the form $N(A)-\mu(A)$, $A\in\mathcal{C}\cap\A_f$. These vectors are therefore $\mathcal{K}$-measurable, thus $\mathcal{C}^*\subset \mathcal{K}$.
\end{proof}
 
 \begin{corollary}\label{corollary:prime}
Let $\left(X^*,\mathcal{A}^{*},\mu^{*},T_{*}\right)$ be a \PaP\ 
suspension. If $T$ is ergodic and has no non-trivial factor, then
$T_{*}$ is prime.
\end{corollary}

\begin{prop}\label{Poissonfactors}
Let $\left(X^{*},\mathcal{A}^{*},\mu^{*},T_{*}\right)$ be a $\mathcal{P}a\mathcal{P}$
suspension and $\mathcal{K}\subset\mathcal{A}^{*}$ a factor, then
$T_{*}$ is relatively weakly mixing over $\mathcal{K}$ if and only
if $\mathcal{K}$ is a Poisson factor.\end{prop}
\begin{proof}
The fact that $T_{*}$ is relatively weakly mixing over $\mathcal{K}$
if it is a Poisson factor was remarked in \cite{Roy07Infinite}. To prove the converse,
assume $T_{*}$ is relatively weakly mixing over $\mathcal{K}$, this
means that the relatively independent joining over $\mathcal{K}$
is ergodic and therefore a Poisson joining thanks to the $\mathcal{P}a\mathcal{P}$
property. Then $\mathcal{K}$ is Poisson thanks to Proposition 4.7
in \cite{Roy07Infinite}.\end{proof}

The notion of semi-simplicity was introduced in \cite{semisimple}:

\begin{definition}
The probability preserving dynamical system $\left(\Omega,\mathcal{F},\mathbb{P},S\right)$ is \emph{semi-simple}
if any ergodic self-joining $\left(\Omega\times\Omega,\mathcal{F}\otimes\mathcal{F},m,S\times S\right)$,
is a relatively weakly mixing extension over $\left(\Omega,\mathcal{F},\mathbb{P},S\right)$
through the projection map.
\end{definition}
As an easy consequence of the preceding result, we get:
\begin{corollary}
A \PaP\ suspension is semi-simple.
\end{corollary}

\begin{prop}
\label{prop:PoissonVariables}
Let $\left(X^{*},\mathcal{A}^{*},\mu^{*},T_{*}\right)$ be a $\mathcal{P}a\mathcal{P}$
suspension and a factor $\mathcal{K}$ generated by random variables
of the form $N\left(A_{i}\right)$, $A_{i}\in\mathcal{A}_{f}$, $i\in I$.
Then $\mathcal{K}$ is a Poisson factor.\end{prop}
\begin{proof}
If $\Phi$ is the conditional expectation corresponding to $\mathcal{K}$,
then as in the proof of Proposition \ref{prop:subPoissonFactor}, it coincides on the first chaos with
$\widetilde{\varphi}$ for some conditional expectation on a $\sigma$-finite
factor $\mathcal{C}\subset\mathcal{A}\mid_{Y}$, where $Y$ is a $T$-invariant
subset of $X$ and $\mathcal{C}^{*}\subset\mathcal{K}$ . Therefore,
the random variables $N\left(A_{i}\right)-\mu\left(A_{i}\right)$
are  in the image of both  $\Phi$ and $\widetilde{\varphi}$, and as
such are $\mathcal{C}^{*}$-measurable. This implies that $\mathcal{K}\subset\mathcal{C}^{*}$
and therefore $\mathcal{K}=\mathcal{C}^{*}$.\end{proof}
\begin{prop}
\label{prop:PaPFactors} A Poisson factor of a \PaP\ suspension is also
\PaP.
\end{prop}
\begin{proof}
Let $\left(Y^{*},\mathcal{B}^{*},\nu^{*},S_{*}\right)$ be a Poisson
factor of a $\mathcal{P}a\mathcal{P}$ suspension $\left(X^{*},\mathcal{A}^{*},\mu^{*},T_{*}\right)$.
An ergodic self-joining of the former can be embedded into an ergodic
self-joining of the latter, which is Poisson by hypothesis. Therefore, by definition of Poisson joinings, we get
another suspension $(Z^*,\mathcal{Z}^{*},\rho^*,R_*)$ in which the two copies of $X^*$ are seen as Poisson factors.
Since Poisson factors of Poisson factors are still Poisson factors, the two copies of $Y^*$ are Poisson factors of $Z^*$, and we get the result.
\end{proof}
 
\begin{theo}\label{theo:PaPandPrime}
Assume that properties~\eqref{P1} and~\eqref{P2} hold for $T$. Then the Poisson suspension $\left(X^{*},{\A}^{*},\mu^{*},T_{*}\right)$ is \PaP $\left(\infty\right)$, prime, mildly mixing and its centralizer is reduced to the powers of $T_{*}$.
\end{theo}

\begin{proof}

By Theorem~\ref{theo:Any-ergodic-countable}, the suspension is \PaP $\left(\infty\right)$.
Primeness comes from Corollary \ref{corollary:prime}, and the fact that properties~\eqref{P1} and~\eqref{P2} imply that $T$ has no non-trivial factor (see~\cite{chaconinfinite}, Section~4). The triviality of the centralizer of $T$ follows also from
properties~\eqref{P1} and~\eqref{P2} (see again~\cite{chaconinfinite}, Section~4), then applying  Proposition \ref{prop:centralizer}, we get that $T_*$ commutes only with its powers.
At last, a transformation is mildly mixing if it has no non-trivial rigid factor, which in the situation of a prime transformation reduces to the property that $T_*$ is not rigid. This is the case, since a rigid transformation has an uncountable centralizer (see \textit{e.g.}~\cite{King86Comweak}).
\end{proof}

We know that properties~\eqref{P1} and~\eqref{P2} imply the triviality of the centralizer of $T$. Hence, for $n\ge2$,  $T^{n}$ never satisfies~\eqref{P1} and~\eqref{P2}, even if $T$ does. We can nevertheless obtain the \PaP\ property for $T^n_*$ when $T$ satisfies properties~\eqref{P1} and~\eqref{P2}. This is a direct application of a lemma we borrow from \cite{LemParThou00Gausselfjoin}.
\begin{lemma}
\label{lem:Ergodic joinings}Let $R$ and $S$ be two commuting ergodic
automorphisms of the probability space $\left(\Omega,\mathcal{F},\mathbb{P}\right)$.
Let $J_{2}^{e}\left(R\right)$ (resp. $J_{2}^{e}\left(S\right)$)
be the set of ergodic self-joinings of $R$ (resp. $S$) and let $F=\overline{\left\langle R\right\rangle }$
(resp. $G=\overline{\left\langle R,S\right\rangle }$) be the closure
of the group generated by $R$ (resp. $R$ and $S$) inside $\text{Aut}\left(\Omega,\mathcal{F},\mathbb{P}\right)$.
If $G/F$ is compact and $J_{2}^{e}\left(S\right)\subset J_{2}^{e}\left(R\right)$,
then $J_{2}^{e}\left(S\right)=J_{2}^{e}\left(R\right)$.\end{lemma}
\begin{corollary}
\label{cor:powers}If $\left(X,\mathcal{A},\mu,T\right)$ satisfies
properties~\eqref{P1} and~\eqref{P2} then for any $n\ge 2$, $T_{*}^{n}$
has the same ergodic self-joinings as $T_{*}$. In particular, $T_{*}^{n}$ is \PaP.
\end{corollary}
\begin{proof}
First observe that $T_{*}^{n}$ is an ergodic Poisson suspension. Indeed,  
$T^{n}$ is a conservative infinite measure
preserving automorphism of $\left(X,\mathcal{A},\mu\right)$ without
$T^n$-invariant set of finite measure: otherwise, if $A$ satisfied $T^{n}A=A$
with $\mu\left(A\right)<+\infty$, then $\bigcup_{k=1}^{n}T^{k}A$
would be a $T$-invariant set of finite measure, which is impossible.

As $C\left(T_*\right)$ is reduced to the powers of $T_*$, $\overline{\left\langle T_{*}^{n}\right\rangle }=\left\{ T_{*}^{kn},\,k\in\mathbb{Z}\right\} $ and $\overline{\left\langle T_{*},T_{*}^{n}\right\rangle }=\left\{ T_{*}^{k},\,k\in\mathbb{Z}\right\} $, hence the quotient is finite. 

To apply Lemma~\ref{lem:Ergodic joinings}, it only remains to check that $J_2^e(T_*)\subset J_2^e(T_*^n)$. Of course, an ergodic self-joining
of $T_{*}$ is a self-joining of $T_{*}^{n}$, but we have to prove it is ergodic. 
By the \PaP\ property for $T_*$, an ergodic self-joining of $T_*$ is Poisson, and thus it is a Poisson self-joining of $T_*^n$.
But a Poisson self-joining of an ergodic suspension is itself ergodic, and therefore we have the
desired inclusion $J_{2}^{e}\left(T_{*}\right)\subset J_{2}^{e}\left(T_{*}^{n}\right)$.
\end{proof}

\begin{prop}
\label{prop:BigMama}
Assume that properties~\eqref{P1} and~\eqref{P2} hold for $T$. Then the Poisson suspension
$$\left(\left(X\times\mathbb{R}_{+}\right)^{*},\left({\A}\otimes\mathcal{B}\right)^{*},\left(\mu\otimes\lambda\right)^{*},\left(T\times \Id \right)_{*}\right)$$
is \PaP.
\end{prop}

\begin{proof}
We denote by $\mathcal{N}$ the Poisson process of intensity $\mu\otimes\lambda$ on $X\times \RR_+$. 
We approach this process by a sequence of Poisson processes discretized on the second coordinate.
For each $n\ge1$, we consider the application $\pi_n$ defined on $X\times\RR_+$ by
\[
  \pi_n(x,t)\egdef\left(x, \lfloor 2^n t\rfloor 2^{-n} \right).
\]
Then, we set $\mathcal{N}_n\egdef (\pi_n)_*(\N)$, which is a Poisson process on $X\times\RR_+$, with 
intensity $\mu\otimes\left( \sum_{j\ge0} 2^{-n} \delta_{j2^{-n}}\right)$. It is therefore concentrated on a countable union
of disjoint copies of $X$, which are the sets $X\times \{j2^{-n}\}$, $j\ge0$. 

Now observe that the  following  convergence holds everywhere on $(X\times\RR_+)_*$: for any continuous function $f:X\times\RR_+\to \RR$, vanishing outside a bounded set,
\[
  \int f\, d\N_n \tend{n}{\infty}\int f\, d\N.
\]
By the dominated convergence theorem, 
\[
  \EE\left[ \exp\left(-\int f\, d\N_n\right) \right] \tend{n}{\infty}\EE\left[ \exp\left(-\int f\, d\N\right) \right].
\]
This is enough to prove the weak convergence of the distribution of $\N_n$ to the distribution of $\N$ (see \cite{DaleyVereJones}, Proposition~11.1.VIII).

Now consider an ergodic self-joining $\gamma$ of 
$$\left(\left(X\times\mathbb{R}_{+}\right)^{*},\left({\A}\otimes\mathcal{B}\right)^{*},\left(\mu\otimes\lambda\right)^{*},\left(T\times \Id \right)_{*}\right), $$
and denote by $\N$ and $\overline\N$ the corresponding Poisson processes, with joint distribution $\gamma$. 
We set $\N_n$ as above, and $\overline\N_n\egdef (\pi_n)_*(\overline\N)$. By the same arguments as above, we prove that the joint distribution $\gamma_n$ of $(\N_n,\overline\N_n)$ converges
weakly to $\gamma$. Note that $\gamma_n$ is a self-joining of order~2 of the Poisson suspension of intensity $\mu\otimes\left( \sum_{j\ge0} 2^{-n} \delta_{j2^{-n}}\right)$, but it can also be interpreted as an infinite ergodic self-joining of the Poisson suspension $\left(X^*,\A^*, (2^{-n}\mu)^*,T_*\right)$. As such, by Theorem~\ref{theo:Any-ergodic-countable}, it is a Poisson joining. Then by Proposition~\ref{prop:inf-divis}, $\gamma_n$ is infinitely divisible. By Proposition 11.2.II in~\cite{DaleyVereJones}, infinite divisibility is closed under weak convergence of distributions, hence $\gamma$ is also infinitely divisible. We conclude by Proposition~\ref{prop:inf-divis} that $\gamma$ is a Poisson joining. 
\end{proof}
 
 \begin{remark}
   In fact it is possible to strengthen the above proposition by proving that the considered Poisson suspension is \PaP$(\infty)$, and moreover that all its
   finite or countable ergodic self-joinings are factors of itself.
 \end{remark}
 
 Applying Propositions~\ref{prop:PoissonVariables} and~\ref{prop:PaPFactors}, we get the following result.

\begin{corollary}\label{PaPfactors}
The factors corresponding to countable self-joinings in Theorem \ref{theo:Any-ergodic-countable} are actually \PaP\ Poisson factors.
\end{corollary}

\section{Disjointness results}
\label{sec:disjointness}

\subsection{Non-disjointness, factors and distal extension}
\label{sec:lemanczyk}
Furstenberg, when introducing joinings and disjointness in \cite{Fur67Disj}, asked whether two non-disjoint systems always possess a non-trivial common factor. 
In~\cite{Rudolph79}, this was shown to be false by Rudolph. However, all counterexamples to Furstenberg's question known so far have the property that one of the 
two non-disjoint system is a factor of a distal extension of the other one. (For definition and properties of distal extensions, we refer \textit{e.g.} to~\cite{Glas03Ergojoin}, Chapter~10.) This led Lema\'nczyk to ask whether the latter property always holds for two non-disjoint systems~\cite{lemanczyk2003}. 
Actually, our Poisson suspensions provide a new counterexample to Furstenberg's question, which also answer Lema\'nczyk's question negatively. Recall that, for any $\alpha>0$, we denote by $T_{*}^{\left(\alpha\right)}$ the Poisson suspension $\left(X^{*},\mathcal{A}^{*},\left(\alpha\mu\right)^{*},T_{*}\right)$. 

\begin{prop}
\label{prop:Talpha_Tbeta}
Assume that properties~\eqref{P1} and~\eqref{P2} hold for $T$. Then for any $\alpha\neq\beta$, $T_{*}^{\left(\alpha\right)}$ and $T_{*}^{\left(\beta\right)}$ are prime and not disjoint. However,  $T_{*}^{\left(\alpha\right)}$ is never a factor of a distal extension of $T_{*}^{\left(\beta\right)}$.
\end{prop}

\begin{lemma}
  \label{lemma:non_isomorphism}
  If $\alpha\neq\beta$, $T_{*}^{\left(\alpha\right)}$ and $T_{*}^{\left(\beta\right)}$ are not isomorphic.
\end{lemma}

\begin{proof}
  Assume that $T_{*}^{\left(\alpha\right)}$ and $T_{*}^{\left(\beta\right)}$ are isomorphic. Then there exists an ergodic joining of these systems which is supported on the graph of an isomorphism. In this joining, we can find two Poisson $T$-point processes $N_\alpha$ and $N_\beta$, of respective intensity $\alpha\mu$ and $\beta\mu$, each of them generating the whole $\sigma$-algebra.
  Then, by Lemma~\ref{lemma:multiple Poisson}, there exist independent Poisson $T$-point processes $N_\alpha^\infty$, $N_\beta^\infty$ and $N_\alpha^k$, $k\in\ZZ$, which are all measurable with respect to $N_\beta$ (and also with respect to $N_\alpha$), such that 
  \[
    N_\alpha = N_\alpha^\infty + \sum_{k\in\ZZ} N_\alpha^k,
  \]
  and\[
    N_\beta = N_\beta^\infty + \sum_{k\in\ZZ} T_*^k\left(N_\alpha^k\right).
  \]
Then $N_\alpha^\infty$ is both measurable with respect to $N_\beta$, and independent of it because it is independent of the family $(N_\beta^\infty, N_\alpha^k,\ k\in\ZZ)$. It follows that $N_\alpha^\infty=0$ a.s. For the same reason, $N_\beta^\infty=0$ a.s., and we get that the intensities of $N_\alpha$ and $N_\beta$ coincide, \textit{i.e.} $\alpha=\beta$. 
\end{proof}

\begin{proof}[Proof of Proposition~\ref{prop:Talpha_Tbeta}]
We already know from Theorem~\ref{theo:PaPandPrime} that the Poisson suspensions $T_{*}^{\left(\alpha\right)}$ and $T_{*}^{\left(\beta\right)}$ are prime. 
Let us see why they are not disjoint. Assume without loss of generality that $0<\alpha<\beta$. In the direct product of $T_{*}^{\left(\alpha\right)}$ and $T_{*}^{\left(\beta-\alpha\right)}$, we have two independent Poisson $T$-point processes $N_\alpha$ and $N_{\beta-\alpha}$, of respective intensities $\alpha\mu$ and $(\beta-\alpha)\mu$. Then $N_\alpha+N_{\beta-\alpha}$ is a Poisson $T$-point process of intensity $\beta\mu$, which is not independent of $N_\alpha$. Hence the distribution of $(N_\alpha,N_\alpha+N_{\beta-\alpha})$ is a joining of $T_{*}^{\left(\alpha\right)}$ and $T_{*}^{\left(\beta\right)}$ which is not the product measure.

Now, we release the assumption $\alpha<\beta$, and we assume there exists an ergodic map $S$ such that $S\rightarrow T_{*}^{\left(\beta\right)}$
is a distal extension and $T_{*}^{\left(\alpha\right)}$ is a factor
of $S$. Then $T_{*}^{\left(\beta\right)}\vee T_{*}^{\left(\alpha\right)}$ appears as
an ergodic joining and a factor of $S$. By Proposition  \ref{PaPfactors}, the
joining is also a \PaP\ suspension. Then
$T_{*}^{\left(\beta\right)}$ is a Poisson factor of $T_{*}^{\left(\beta\right)}\vee T_{*}^{\left(\alpha\right)}$
and as such the extension $T_{*}^{\left(\beta\right)}\vee T_{*}^{\left(\alpha\right)}\rightarrow T_{*}^{\left(\beta\right)}$
is relatively weakly mixing by Proposition \ref{Poissonfactors}. Therefore we have the following sequence
of extensions
\[
S\rightarrow T_{*}^{\left(\beta\right)}\vee T_{*}^{\left(\alpha\right)}\rightarrow T_{*}^{\left(\beta\right)}
\]

But, as the extension $S\rightarrow T_{*}^{\left(\beta\right)}$ is distal, $T_{*}^{\left(\beta\right)}\vee T_{*}^{\left(\alpha\right)}\rightarrow T_{*}^{\left(\beta\right)}$ cannot be relatively weakly mixing unless it is an isomorphism (see Proposition 10.14 in \cite{Glas03Ergojoin}). Then this implies that $T_{*}^{\left(\alpha\right)}$ is a factor of $T_{*}^{\left(\beta\right)}$. Since the latter is prime by Theorem~\ref{theo:PaPandPrime},  $T_{*}^{\left(\alpha\right)}$ and $T_{*}^{\left(\beta\right)}$ are isomorphic. But by Lemma~\ref{lemma:non_isomorphism}, this happens only if $\alpha=\beta$. 
\end{proof}

\subsection{General results}
\label{sec:general_disjointness}

\begin{definition}
  A \emph{measurable law of large numbers} for a conservative, ergodic, measure preserving dynamical system $(X,\A,\mu,T)$ is a measurable 
  function $L:\{0,1\}^\NN\to[0,\infty]$ such that
  for all $B\in\A$, for $\mu$-almost every $x\in X$,
  \[ 
    L\left(\ind{B}(x),\ind{B}(Tx),\ldots\right) = \mu(B).
  \]
\end{definition}

\begin{lemma}
\label{lemma:XcroixC}
Let $(X,\A,\mu,T)$ be a conservative, ergodic, measure preserving dynamical system, and assume that it admits a   
measurable law of large numbers. Let $\mathcal{L}$ be a $\sigma$-finite factor of the product dynamical system 
\[
  (X\times \RR_+,\A\otimes\B,\mu\otimes\lambda,T\times\Id).
\]
Then there exists $C\subset\RR_+$ with $0<\lambda(C)<\infty$, such that $X\times C\in\mathcal{L}$.
\end{lemma}
\begin{proof}
  Since $\mathcal{L}$ is $\sigma$-finite, there exists $B\in\mathcal{L}$ such that $0<\mu\otimes\lambda(B)<\infty$.
  For each $t\in\RR_+$, let us consider
  \[
    B_t \egdef \{x\in X: (x,t)\in B\}.
  \]
  Denote by $L$ a measurable law of large numbers for $(X,\A,\mu,T)$. Then, for $\mu\otimes\lambda$-almost every $(x,t)\in X\times\RR_+$, we have
  \[
    L\Bigl(\left(\ind{B}(T^kx,t)\right)_{k\ge 0}\Bigr) = \mu(B_t).
  \]
  This ensures that the map $(x,t)\mapsto\mu(B_t)$ is $\mathcal{L}$-measurable. In particular, for any $\varepsilon>0$, the set
  $\{(x,t): \mu(B_t)\ge\varepsilon\}$ is $\mathcal{L}$-measurable. This set is of the form $X\times C$ for $C\subset\RR_+$. We
  have 
  \[\infty > \mu\otimes\lambda(B) = \int_{\RR_+} \mu(B_t)\, d\lambda(t) \ge \varepsilon \lambda(C), \]
  and choosing $\varepsilon$ small enough, we have $\lambda(C)>0$.
\end{proof}

\begin{definition}
  A  conservative, ergodic, measure preserving dynamical system $(X,\A,\mu,T)$ is \emph{rationally ergodic} if there exists 
  a set $B\in\A$, $0<\mu(B)<\infty$, and a constant $M>0$ such that, for any $n\ge1$, 
  \[
    \int_B \left(\sum_{0\le j\le n-1}\ind{B}(T^jx)\right)^2 \, d\mu(x) \le M \left(\int_B \sum_{0\le j\le n-1}\ind{B}(T^jx) \, d\mu(x)\right)^2.
  \]
\end{definition}

According to Theorem~3.3.1 in~\cite{Aaronson}, a measurable law of large numbers exists for $T$ as soon as $T$ is rationally ergodic, which is 
the case of the nearly finite Chacon transformation (see~\cite{nfc}). 
Observe however that properties~\eqref{P1} and~\eqref{P2} alone imply the existence of a law of large numbers, but it happens that the question of its measurability remains open without rational ergodicity.

The following proposition applies therefore to the case of the Poisson suspension over the
nearly finite Chacon transformation.

\begin{prop}
\label{prop:Disjonction}Assume that properties~\eqref{P1} and~\eqref{P2} hold for $T$, and that $T$ admits a measurable law of large numbers. If a system $\left(Y,\mathcal{B},\nu,S\right)$
is not disjoint from some $n$-order ergodic self-joining of $\left(X^{*},{\A}^{*},\mu^{*},T_{*}\right)$,
then it possesses $\left(X^{*},{\A}^{*},\left(\alpha\mu\right)^{*},T_{*}\right)$
as a factor, for some $\alpha>0$.
\end{prop}
\begin{proof}
From a result of \cite{LemParThou00Gausselfjoin}, if $\left(Y,\mathcal{Y},\nu,S\right)$
is not disjoint from an ergodic $n$-order self-joining of $\left(X^{*},\mathcal{A}^{*},\mu^{*},T_{*}\right)$,
then it possesses a common non trivial factor with a countably infinite
self-joining of it. However, an ergodic countably infinite self-joining
of this $n$-order self-joining is nothing else than an ergodic countably
infinite self-joining of $\left(X^{*},\mathcal{A}^{*},\mu^{*},T_{*}\right)$.
This common factor is therefore a factor of $\left(\left(X\times\mathbb{R}_{+}\right)^{*},\left(\mathcal{A}\otimes\mathcal{B}\right)^{*},\left(\mu\otimes\lambda\right)^{*},\left(T\times \Id \right)_{*}\right)$
by Theorem \ref{theo:Any-ergodic-countable}. Since the latter suspension
is \PaP\ by Proposition~\ref{prop:BigMama}, this factor itself contains a Poisson
factor of $\left(\mathcal{A}\otimes\mathcal{B}\right)^{*}$ by Proposition \ref{prop:subPoissonFactor}.
Therefore there exists a $(T\times \Id )$-invariant subset $L\subset X\times\mathbb{R}_{+}$
and a $\sigma$-finite factor $\mathcal{L}$ of the restricted system
such that we have the following factor relationship:
\[
\begin{array}{c}
\left(Y,\mathcal{Y},\nu,S\right)\\
\downarrow\\
\left(L^{*},\mathcal{L}^{*},\left(\mu\otimes\lambda\right)_{\mid L}^{*},\left(T\times \Id \right)_{*}\right).
\end{array}
\]
Using Lemma~\ref{lemma:XcroixC}, we get the existence of $C\subset\RR_+$, with $0<\lambda(C)<\infty$, such that $X\times C\in \mathcal{L}$. Passing if necessary to 
another factor, we can therefore assume that $L$ is of the form $X\times C$, with $0<\lambda(C)<\infty$.

Observe now that the dynamical system $(L,\A\otimes\B|_C,\mu\otimes\lambda|_C,T\times\Id)$ admits both systems $(L,\mathcal{L},\mu\otimes\lambda|_C,T\times\Id)$ and 
$(X,\A,\lambda(C)\mu,T)$ as factors. It therefore defines a joining of these systems, and by Proposition~4.5 in~\cite{chaconinfinite}, $(X,\A,\lambda(C)\mu,T)$ is a factor of $(L,\mathcal{L},\mu\otimes\lambda|_C,T\times\Id)$. 

Passing to Poisson suspensions and going up the chain of factors up
to $\left(Y,\mathcal{Y},\nu,S\right)$, we obtain our result.

\end{proof}

\begin{prop}
\label{prop:disjointness-alpha}Assume that properties~\eqref{P1} and~\eqref{P2} hold for $T$, and that $T$ admits a measurable law of large numbers. If a system $\left(Y,\mathcal{Y},\nu,S\right)$
is disjoint from $\left(X^{*},{\A}^{*},\mu^{*},T_{*}\right)$, 
then it is disjoint from $\left(X^{*},{\A}^{*},\left(\alpha\mu\right)^{*},T_{*}\right)$
for any $\alpha>0$.
\end{prop}

\begin{proof}
Assume $\left(Y,\mathcal{Y},\nu,S\right)$ is disjoint from $\left(X^{*},{\A}^{*},\mu^{*},T_{*}\right)$, but 
that there exists $\alpha>0$ such that it is not disjoint from $\left(X^{*},{\A}^{*},\left(\alpha\mu\right)^{*},T_{*}\right)$. 
Thanks to Proposition~\ref{prop:Disjonction}, there
exists $\beta>0$ such that $\left(X^{*},{\A}^{*},\left(\beta\mu\right)^{*},T_{*}\right)$
is a factor of $\left(Y,\mathcal{Y},\nu,S\right)$ (and $\beta\neq1$ by assumption). This means that there exists a Poisson $T$-point process $N_{1}$ 
of intensity $\beta\mu$ defined on $\left(Y,\mathcal{Y},\nu,S\right)$. 

If $\beta<1$, let us consider the direct product 
$$\left(Y\times X^{*},\mathcal{Y}\otimes{\A}^{*},\nu\otimes\left(\left(1-\beta\right)\mu\right)^{*},S\times T_{*}\right).$$
In this product, there exists a Poisson $T$-point process $N_{2}$ of intensity $\left(1-\beta\right)\mu$
independent of the whole system $\left(Y,\mathcal{Y},\nu,S\right)$. In particular, $N_1$ and $N_2$ are independent, thus 
$N_{1}+N_{2}$ defines a Poisson $T$-point process of intensity
$\mu$, which is independent of $\left(Y,\mathcal{Y},\nu,S\right)$
by disjointness. This is absurd, as this would imply that $N_{1}$
and $N_{1}+N_{2}$ are independent. This is obviously false, since $N_{1}\le N_{1}+N_{2}$.

If $\beta>1$, observe that $\left(X^{*},{\A}^{*},\left(\beta\mu\right)^{*},T_{*}\right)$ is both a factor 
of $\left(Y,\mathcal{Y},\nu,S\right)$ and of the direct product 
\[
\left(X^{*}\times X^{*},{\A}^{*}\otimes{\A}^{*},\mu^{*}\otimes\left(\left(\beta-1\right)\mu\right)^{*},T_{*}\times T_{*}\right).
\]
In this product, we have two Poisson $T$-point processes $N_{3}$ and
$N_{4}$ of respective intensities $\mu$ and $\left(\beta-1\right)\mu$,
such that $N_{3}+N_{4}$ is a Poisson $T$-point process of intensity
$\beta\mu$ that corresponds to the factor $\left(X^{*},{\A}^{*},\left(\beta\mu\right)^{*},T_{*}\right)$.
We can then form the relatively independent joining of $\left(Y,\mathcal{Y},\nu,S\right)$
and this direct product 
over $\left(X^{*},{\A}^{*},\left(\beta\mu\right)^{*},T_{*}\right)$. In this scheme, we have $N_1=N_3+N_4$ almost surely, 
hence $N_3+N_4$ is measurable with respect to  the $\sigma$-algebra
$\mathcal{Y}$. But $N_3$ is independent of $\mathcal{Y}$ by disjointness, leading to the same contradiction that $N_3$ is independent of $N_3+N_4$.
\end{proof}

\begin{prop}
\label{prop:disjointness of any order}
Assume that properties~\eqref{P1} and~\eqref{P2} hold for $T$, and that $T$ admits a measurable law of large numbers. 
If a system $\left(Y,\mathcal{Y},\nu,S\right)$
is disjoint from $\left(X^{*},{\A}^{*},\mu^{*},T_{*}\right)$, then it is disjoint from any self-joining (of any order) of this 
Poisson suspension.
\end{prop}

\begin{proof}
If there exists a self-joining of $\left(X^{*},{\A}^{*},\mu^{*},T_{*}\right)$
not disjoint from $\left(Y,\mathcal{Y},\nu,S\right)$, then, by Proposition~\ref{prop:Disjonction}, $\left(Y,\mathcal{Y},\nu,S\right)$
possesses $\left(X^{*},{\A}^{*},\left(\alpha\mu\right)^{*},T_{*}\right)$
as a factor, for some $\alpha>0$. However, by Proposition~\ref{prop:disjointness-alpha}, this implies that $\left(Y,\mathcal{Y},\nu,S\right)$
is not disjoint from $\left(X^{*},{\A}^{*},\mu^{*},T_{*}\right)$.
\end{proof}

\begin{corollary}
Assume that properties~\eqref{P1} and~\eqref{P2} hold for $T$, and that $T$ admits a measurable law of large numbers. A system $\left(Y,\mathcal{Y},\nu,S\right)$ is disjoint from $\left(X^{*},{\A}^{*},\mu^{*},T_{*}\right)$
if and only if it is disjoint from 
$$\left(\left(X\times\mathbb{R}_{+}\right)^{*},\left({\A}\otimes\mathcal{B}\right)^{*},\left(\mu\otimes\lambda\right)^{*},\left(T\times \Id \right)_{*}\right).$$
\end{corollary}

\begin{proof}
First observe that we can view 
\[
  \left(\left(X\times\mathbb{R}_{+}\right)^{*},\left({\A}\otimes\mathcal{B}\right)^{*},\left(\mu\otimes\lambda\right)^{*},\left(T\times \Id \right)_{*}\right)
\]
as the inverse limit of the direct products 
\[  
\left(\left(X^{*}\right)^{2^{2n}},\left({\A}^{*}\right)^{\otimes2^{2n}},\left(\left(\frac{1}{2^n}\mu\right)^{*}\right)^{\otimes2^{2n}},\left(T_{*}\right)^{\times2^{2n}}\right)
\] as $n$ tends to infinity.
Assume that $\left(Y,\mathcal{Y},\nu,S\right)$ is disjoint from $\left(X^{*},{\A}^{*},\mu^{*},T_{*}\right)$. Then by Proposition~\ref{prop:disjointness-alpha}, it is disjoint from $\left(X^{*},{\A}^{*},\left(\frac{1}{2^n}\mu\right)^{*},T_{*}\right)$, and by Proposition~\ref{prop:disjointness of any order}, it is also disjoint from any self-joining of this system.  Then, passing to the inverse limit, 
we conclude that it is also disjoint from 
\[
  \left(\left(X\times\mathbb{R}_{+}\right)^{*},\left({\A}\otimes\mathcal{B}\right)^{*},\left(\mu\otimes\lambda\right)^{*},\left(T\times \Id \right)_{*}\right).
\]
The converse is obvious since the above system admits $\left(X^{*},{\A}^{*},\mu^{*},T_{*}\right)$ as a factor.
\end{proof}

\subsection{Disjointness from classical classes of dynamical systems}
\label{sec:disjointness_from_classical_classes}

There already exist general disjointness results that concern Poisson suspensions: it is proved in~\cite{LemParRoy2011} that Poisson suspensions are disjoint from transformations that enjoy the \emph{joining primeness property}, such as distally simple transformations.
We can nevertheless obtain stronger disjointness results for the suspensions we are interested in.

\begin{theo}
\label{theo:disjointness_local_rank1}
If $T$ satisfies properties~\eqref{P1} and~\eqref{P2}, and admits a measurable law of large numbers, then $\left(X^{*},{\A}^{*},\mu^{*},T_{*}\right)$ is disjoint from
any rank one transformation.
\end{theo}

\begin{proof}
If a rank one transformation is not disjoint from $\left(X^{*},{\A}^{*},\mu^{*},T_{*}\right)$, then by Proposition~\ref{prop:Disjonction} it possesses $\left(X^{*},{\A}^{*},\left(\alpha\mu\right)^{*},T_{*}\right)$ as a factor for some $\alpha>0$.  But a factor of a rank one transformation is
also of rank one.

On the other hand, $\left(X^{*},{\A}^{*},\left(\alpha\mu\right)^{*},T_{*}\right)$
is mildly mixing thanks to Theorem~\ref{theo:PaPandPrime}, and we know from Proposition $11$ in~\cite{ParreauRoy} 
that a non-rigid Poisson suspension is not of rank one.
\end{proof}

\begin{remark}
  According to Ryzhikov~\cite{LocalRank}, a non-rigid Poisson suspension is in fact not even of local rank one. Thus the above theorem extends to local rank one transformations.
\end{remark}

We now turn to disjointness from Gaussian dynamical system, about which we first recall a few facts.
A dynamical system $\left(\Omega,\mathcal{F},\mathbb{P},S\right)$ 
is said to be \emph{standard Gaussian} if there exists some measurable function $f$ of zero mean
defined on $\Omega$ such that $X_{n}:=f\circ S^{n}$ defines a Gaussian
stationary process that generates $\mathcal{F}$. Up to measurable isomorphism, such a dynamical system is completely
identified by the spectral measure $\sigma$ of $f$ on $\mathbb{T}$:
\[
\left\langle X_{0},X_{n}\right\rangle _{L^{2}\left(\mathbb{P}\right)}=\widehat{\sigma}\left(n\right).\]
As in the Poisson case, $L^{2}\left(\mathbb{P}\right)$ admits a Fock
space representation
\[
L^{2}\left(\mathbb{P}\right)\simeq\mathbb{C}\oplus L^{2}\left(\sigma\right)\oplus L^{2}\left(\sigma\right)^{\odot2}\oplus\cdots\oplus L^{2}\left(\sigma\right)^{\odot n}\oplus\cdots
\]
Therefore, $L^{2}\left(\mathbb{P}\right)$ admits a decomposition into
(Gaussian) chaos $\left\{ C_{n}\right\} _{n\ge0}$ and the maximal
spectral type of $U_{S}$ on $C_{n}$ is $\sigma^{*n}$. (For a detailed presentation of the spectral analysis of Gaussian dynamical systems, we refer \textit{e.g.} to~\cite{CFS}, Chapter~14.)

A particularly interesting situation for us arises when a Gaussian system (or a Poisson suspension) has simple spectrum. Indeed, it then enjoys the following property, presented in the form of a proposition which can be found in a more general form in \cite{KulagaParreau}:
\begin{prop}
  \label{prop:simple_spectrum_Gaussian}
  If a standard Gaussian dynamical system (resp. a Poisson suspension) has simple spectrum, then for any pair $m_{1}$, $m_{2}$ of continuous measures on $\TT$, the spectral measure $\sigma$ of the Gaussian process (resp. the maximal spectral type of the base of the suspension)
  satisfies $\sigma\perp m_{1}*m_{2}$.
\end{prop}

As an application, we get the following result which will be useful for our purposes:

\begin{prop}
\label{prop:PoissonFactorOfGaussian}
  A Poisson suspension is never a factor of a standard Gaussian dynamical system with simple spectrum.
\end{prop}

\begin{proof}
  Assume that $\left(Y,\mathcal{B},\nu,S\right)$ is a standard Gaussian with simple spectrum, and that it admits the Poisson suspension $(Z^*, \mathcal{C}^*, \rho^*, R_*)$ as a factor. Let $\eta$ be the spectral measure of the generating Gaussian process, and let $\sigma$ be the maximal spectral type of $(Z,\mathcal{C}, \rho, R)$, which is also the maximal spectral type of the action of $U_{R^*}$ on the first Poissonian chaos of the suspension. Since  $R^*$ has simple spectrum (as a factor of the simple spectrum system $S$) Proposition~\ref{prop:simple_spectrum_Gaussian} applies to $\sigma$. Now, take $f\neq 0$ in the first Poissonian chaos of the suspension: its spectral measure is absolutely continuous with respect to $\sigma$,
 and by~Proposition~\ref{prop:simple_spectrum_Gaussian}, it is singular with respect to $\eta^{*n}$ for any $n\ge2$. Identifying $L^2(\rho^*)$ with a subspace of $L^2(\nu)$, we conclude that $f$ has to be in the first Gaussian chaos of $L^2(\nu)$. But this is impossible, because the first Gaussian chaos contains exclusively (complex) Gaussian random variables, whereas the first Poissonian chaos of a suspension contains no such variables, the zero vector aside (see \textit{e.g.}~\cite{Roy07Infinite}).
\end{proof}

Now we can state and prove the following result on disjointness between our Poisson suspensions and standard Gaussian dynamical systems:

\begin{theo}
\label{theo:disjointnessPoissonGauss}
If $T$ satisfies properties~\eqref{P1} and~\eqref{P2}, and admits a measurable law of large numbers,  then $\left(X^*,\mathcal{A}^{*},\mu^{*},T_{*}\right)$
is disjoint from any standard Gaussian system.
\end{theo}

\begin{proof}
Let $\left(Y,\mathcal{B},\nu,S\right)$ a standard Gaussian, and assume it is not disjoint
from $\left(X,\mathcal{A}^{*},\mu^{*},T_{*}\right)$.
Applying Proposition~\ref{prop:Disjonction}, we get the existence of a factor sub-$\sigma$-algebra 
$\mathcal{C}\subset\mathcal{B}$ such that the action of $S$ on $\mathcal{C}$ is isomorphic to $\left(X^*,\mathcal{A}^{*},\left(\alpha\mu\right)^{*},T_{*}\right)$
for some $\alpha>0$. 
As in the proof of Proposition 17 in~\cite{LemParThou00Gausselfjoin}, there
exists $S^{\prime}\in C\left(S\right)$ which is a standard Gaussian
system with simple spectrum. 
Set $\mathcal{C}^{\prime}:=\vee_{n\in\mathbb{Z}}S^{\prime-n}\mathcal{C}$,
then $\mathcal{C}^{\prime}$ is a factor $\sigma$-algebra of $S^{\prime}$. 
Observe that for any $n\in\ZZ$, $S^{\prime-n}\mathcal{C}$ is also a factor $\sigma$-algebra of $S$, and that the action of $S$ on $S^{\prime-n}\mathcal{C}$ is also isomorphic to $\left(X^*,\mathcal{A}^{*},\left(\alpha\mu\right)^{*},T_{*}\right)$. It follows that the action of 
$S$ on $\vee_{n\in\mathbb{Z}}S^{\prime-n}\mathcal{C}$ defines an ergodic countable self-joining of $\left(X^*,\mathcal{A}^{*},\left(\alpha\mu\right)^{*},T_{*}\right)$
and, as such, is isomorphic to a $\mathcal{P}a\mathcal{P}$ suspension $\left(Z^{*},\mathcal{Z}^{*},\rho^{*},T_{*}^{\infty}\right)$ by Corollary~\ref{PaPfactors}. It follows that
$S_{\mid\mathcal{C}^{\prime}}^{\prime}$ can be considered as an element
of $C\left(T_{*}^{\infty}\right)$, but since the suspension is $\mathcal{P}a\mathcal{P}$,
from Proposition \ref{prop:centralizer}, $S_{\mid\mathcal{C}^{\prime}}^{\prime}=R_{*}$
for some automorphism $R$ of $\left(Z,\mathcal{Z},\rho\right)$ commuting
with $T^{\infty}$. Then we get the Poisson suspension $R^*$ as a factor of the standard Gaussian dynamical system $S^{\prime}$ which has simple spectrum, and this contradicts Proposition~\ref{prop:PoissonFactorOfGaussian}.
\end{proof}

\section{Conclusion}

Our work raises several questions, among which a natural one is the following: is
it possible to obtain the same results, assuming only moments of order
2 for the point process?
We can also ask ourselves whether we could obtain similar results with the base transformation 
$T$ having uncountably many ergodic self-joinings (for example with an uncountable
centralizer)? This would require very different techniques, as our
proofs strongly rely on the fact that $T$ possesses a countable set of ergodic self-joinings.
More generally, it would be interesting to know if the $\mathcal{P}a\mathcal{P}$ property is widespread among Poisson suspensions, or if it is a rare feature.

\bibliography{sushi}

\end{document}